\documentclass[11pt,twoside]{amsart}
\usepackage{amsmath,amssymb,amsthm}
\usepackage{color}
\usepackage{cite}

\newtheorem{thm}{Theorem}[section]

 \newtheorem{cor}{Corollary}[section]
 \newtheorem{lem}{Lemma}[section]
 \newtheorem{prop}{Proposition}[section]
 \newtheorem{defn}{Definition}[section]
\newtheorem{rem}{Remark}[section]

\def\Id{{\rm Id}\,}
\def\d{\partial}
\def\ddj{\dot \Delta_j}

\def\tilde{\widetilde}
\def\hat{\widehat}

\newcommand\R{\mathbb{R}}

\newcommand\Z{\mathbb{Z}}

\newcommand{\N}{\mathbb{N}}

\renewcommand{\div}{\mbox{\rm div}\;\!}

\begin{document}
\title[compressible Navier-Stokes-Poisson system]{A sharp time-weighted inequality for the
compressible Navier-Stokes-Poisson system in the critical $L^{p}$ framework}

\author{Wei Xuan Shi}
\address{Department of Mathematics,  Nanjing
University of Aeronautics and Astronautics,
Nanjing 211106, P.R.China,}
\email{wxshi168@163.com}

\author{Jiang Xu}
\address{Department of Mathematics,  Nanjing
University of Aeronautics and Astronautics,
Nanjing 211106, P.R.China,}
\email{jiangxu\underline{ }79math@yahoo.com}


\begin{abstract}
The compressible Navier-Stokes-Poisson system takes the form of usual Navier-Stokes equations coupled with the self-consistent Poisson equation, which is used to simulate the transport of charged particles under the electric field of electrostatic potential force. In this paper, we focus on the large time behavior
of global strong solutions in the $L^{p}$ Besov spaces of critical regularity. By exploring the dissipative effect arising from Poisson potential, we posed the new regularity assumption of low frequencies and then establish a \textit{sharp} time-weighted inequality, which leads to the optimal time-decay estimates of the solution. Indeed, we see that the decay of density is faster at the half rate than that of velocity, which is a different ingredient in comparison with the situation of usual Navier-Stokes equations. Our proof mainly depends on tricky and non classical Besov product estimates with respect to various Sobolev embeddings.
\end{abstract}

\maketitle

\section{Introduction} \setcounter{equation}{0}
It is well known that the compressible Navier-Stokes-Poisson system can be used to simulate the transport of charged particles in semiconductor devices
under the influence of electric fields (see for example \cite{MRS} for more explanations). The barotropic compressible Navier-Stokes-Poisson system is given by
\begin{equation} \label{Eq:1.1}
\left\{
\begin{array}{l}
\partial _{t}\varrho +\mathrm{div}\left( \varrho u\right) =0, \\
\partial _{t}( \varrho u) +\mathrm{div}\left( \varrho u\otimes u\right) +\nabla P\left(\varrho\right)=\mathrm{div}\left( 2\mu\, D(u) +\lambda\,\mathrm{div}\,u\,\mathrm{Id}\right)- \varrho \nabla \psi,\\
-\Delta \psi=\varrho-\varrho_{\infty}.
\end{array}
\right.
\end{equation}
Here,  $\varrho =\varrho (t,x)\in \R_{+}$ and $u=u(t,x)\in \mathbb{R}^{d} (d\geq2)$ are the unknown functions on $[0, +\infty)\times \mathbb{R}^{d}$, which stands for the density and velocity field of charged particles, respectively.  $\psi=\psi(t,x)$ denotes the electrostatic potential force.
The barotropic assumption means that the pressure $P$ is given suitably smooth function of $\varrho$.
The notation $D(u)\triangleq\frac{1}{2}(D_{x} u+{}^T\!D_{x} u)$ stands for  the {\it deformation tensor}, and $\nabla$ and $\div$ are the gradient and divergence operators with respect to the space variable. The Lam\'{e} coefficients $\lambda$ and $\mu$
(the \emph{bulk} and \emph{shear viscosities}) are density-dependent functions, which are supposed to be smooth functions of density and to satisfy $\mu>0$ and $\lambda +2\mu>0$. The initial condition of System \eqref{Eq:1.1} is prescribed by
\begin{equation}\label{Eq:1.2}
\left(\varrho,u\right)|_{t=0}=\left(\varrho _{0}(x),u_{0}(x)\right),\  x\in \R^{d}.
\end{equation}

In this paper, we focus on the Cauchy problem \eqref{Eq:1.1}-\eqref{Eq:1.2} and investigate the asymptotic behavior of global solutions, as initial data are the perturbation of constant equilibrium $\left(\varrho _{\infty},0\right)$ with $\varrho _{\infty}>0$. First of all, let us take a look at the spectral analysis briefly, which was made by Li, Matsumura \& Zhang \cite{LMZ}. It is not difficult to check that the linearized system (around the equilibrium) reads
\begin{equation*}
\left\{
\begin{array}{l}
\partial _{t} a +\varrho_{\infty}\mathrm{div}\,u=0, \\
\partial _{t}u-\frac{\mu_{\infty}}{\varrho_{\infty}}\Delta u-\frac{\mu_{\infty+\lambda_{\infty}}}{\varrho_{\infty}}\nabla \mathrm{div}u
+\frac{P'(\varrho_{\infty})}{\varrho_{\infty}}\nabla a +\nabla (-\Delta)^{-1} a=0
\end{array}
\right.
\end{equation*}
with $a=\varrho-\varrho_{\infty}$, $\mu_{\infty}=\mu(\varrho_{\infty})$ and $\lambda_{\infty}=\lambda(\varrho_{\infty})$. It follows from \cite{LMZ} that
\begin{equation*}
\hat{a}(t,\xi)\sim \left\{
\begin{array}{l}
e^{-c|\xi|^{2}t}\left(|\hat{a}_{0}|+\left|\xi\right| |\hat{u}_{0}|\right), |\xi|\rightarrow 0;\\
e^{-ct}\left(|\hat{a}_{0}|+|\hat{u}_{0}|\right), \, |\xi|\rightarrow\infty
\end{array}
\right.
\end{equation*}
and
\begin{equation*}
\hat{u}(t,\xi)\sim \left\{
\begin{array}{l}e^{-c|\xi|^{2}t}\left(|\xi|^{-1} |\hat{a}_{0}|+|\hat{u}_{0}|\right), |\xi|\rightarrow 0;\\
e^{-ct}\left(|\hat{a}_{0}|+|\hat{u}_{0}|\right),  |\xi|\rightarrow\infty,
\end{array}
\right.
\end{equation*}
where $c>0$ is a generic constant which depends only on $\varrho_{\infty}, \mu_{\infty}$ and $\lambda_{\infty}$. If the initial data
$\left(\varrho_{0}-\varrho_{\infty}, u_{0}\right)\in L^{2}(\mathbb{R}^{3})\cap L^{q}(\mathbb{R}^{3})(q\in [1,2])$, then one can
conclude the optimal decay rates of $L^{2}$-$L^{q}$ type:
\begin{equation}\label{Eq:1.3}
\|(\varrho-\varrho_{\infty})(t)\|_{L^{2}}\leq C \langle t\rangle^{-\frac{3}{2}(\frac{1}{q}-\frac{1}{2})}, \
\|u(t)\|_{L^{2}}\leq C \langle t\rangle^{-\frac{3}{2}(\frac{1}{q}-\frac{1}{2})+\frac{1}{2}} \ \mbox{with}\  \langle t \rangle\triangleq \sqrt{1+t^{2}}.
\end{equation}
Obviously, due to the presence of the electric potential, there is a decay difference between the density and velocity field. However, we see that
the $L^{2}$ norm of the velocity $u$ grows in time at the rate $\langle t\rangle^{\frac{1}{2}}$ possibly if taking $q=2$ in \eqref{Eq:1.3}, which seems to be
a contradiction with the dissipativity of Navier-Stokes-Poisson equations. To eliminate the obstacle, a stronger assumption, say $\Lambda^{-1}(\varrho_{0}-\varrho_{\infty})\in L^{2}$, was posed in \cite{CD2,WYJ}, which leads to the substituting decay in comparison with \eqref{Eq:1.3}:
\begin{equation*}
\left\|\left(\varrho-\varrho_{\infty}\right)(t)\right\|_{L^{2}}\leq C \langle t\rangle^{-\frac{3}{2}(\frac{1}{q}-\frac{1}{2})-\frac{1}{2}}
\ \ \hbox{and}\ \
\left\|u(t)\right\|_{L^{2}}\leq C \langle t\rangle^{-\frac{3}{2}(\frac{1}{q}-\frac{1}{2})}.
\end{equation*}
We would like to mention that $\Lambda^{-1}(\varrho_{0}-\varrho_{\infty})\in L^{2}$ is equivalent to $\nabla\psi_{0}\in L^{2}$, which is a natural condition from the point of view of the energy approach. The reader is referred to \cite{X1} for the similar situation. In short, the electric potential may enhance the dissipation of density such that it enjoys addition half time-decay rate than velocity, with the \textit{relatively} stronger assumption than \cite{LMZ}.
Up to now, there are a number of efforts dedicated to the large-time behavior of solutions to \eqref{Eq:1.1}-\eqref{Eq:1.2}, see \cite{LMZ,LZ,WYJ,WW1,WW2} and references therein.

Let us emphasize that above results mentioned are established in Sobolev spaces with higher regularity. Our aim of this paper is to
prove (more accurate) decay estimates in the critical regularity Besov spaces, which is close to the framework of weak solutions.  Indeed, observe that \eqref{Eq:1.1} is obviously invariant for all $\lambda>0$ by the following transformation
\begin{equation*}
\varrho ( t,x) \rightsquigarrow \varrho \left( \lambda^{2}t,\lambda x\right) ,\ \ u(t,x) \rightsquigarrow \lambda u\left( \lambda^{2}t,\lambda x\right) ,\ \ \psi(t,x) \rightsquigarrow \lambda^{-2}\psi\left( \lambda^{2}t,\lambda x\right),
\end{equation*}
if neglecting the lower order pressure term and electronic field term. Therefore, it is so interesting to find a functional space as large as possible such that the global existence and uniqueness holds. To the best of our knowledge, the idea is classical now. Motivated by those works on the incompressible Navier-Stokes equations \cite{CM,FK, KY}, Danchin \cite{DR1} considered compressible Navier-Stokes equations and showed the global-in-time existence of strong solutions in $L^{2}$ critical Besov spaces. Subsequently,  Charve \& Danchin \cite{CD1}, Chen, Miao \& Zhang \cite{CMZ} independently extended that result to more general $L^{p}$ critical spaces. In \cite{HB}, Haspot obtained essentially the same results as in \cite{CD1,CMZ} by using an elementary energy approach which is based on the use of Hoff's viscous effective flux in \cite{HD}. In the recent work \cite{DX}, Danchin and the second author established the optimal time decay estimates in more general $L^{p}$ framework.

However, there are few results to our knowledge on the global existence and large-time behavior of solutions to \eqref{Eq:1.1}-\eqref{Eq:1.2} in spatially critical spaces. Hao and Li \cite{HL} established the global existence of small strong solution to \eqref{Eq:1.1} in the $L^{2}$ critical hybrid Besov space and in any dimension $d\geq3$. Recently, Chikami and Danchin \cite{CD2} constructed the unique global strong solution near constant equilibrium in the general $L^{p}$ critical Besov spaces, which contains the case of dimension $d=2$. For simplicity, they assumed that viscosities $\mu$ and $\lambda$ were constant only. Indeed, their result remains true in case of the density-dependent smooth functions. For convenience, we state it as follows (the reader is also referred to \cite{CD2}).

\begin{thm}\label{Thm1.1} Let $d\geq2$ and $p$ fulfill
\begin{equation}\label{Eq:1.5}
2\leq p\leq \min\left(4,2d/(d-2)\right) \ \hbox{and}, \ \hbox{additionally}, \ p\neq 4 \ \hbox{if} \ d=2.
\end{equation}
Assume that $P'\left(\varrho_{\infty} \right)>0$ and that \eqref{Eq:1.2} is satisfied.
There exists a small positive constant $c=c\left(p,d,\mu,\lambda, P,\varrho_{\infty}\right)$ and a universal integer $j_{0}\in
\mathbb{N}$ on $\mu$ and $\lambda$ such that if $a_{0}\triangleq \left(\varrho_{0}-\varrho_{\infty}\right) \in \dot{B}_{p,1}^{\frac {d}{p}}$, if $u_{0}\in \dot{B}_{p,1}^{\frac {d}{p}-1}$ and if in addition $\left(a_{0}^{\ell},\nabla u_{0}^{\ell} \right)\in \dot{B}_{2,1}^{\frac {d}{2}-2}$ with
\begin{equation*}
\mathcal{E}_{p,0}\triangleq \left\|\left( a_{0},\nabla u_{0}\right) \right\|_{\dot{B}_{2,1}^{\frac {d}{2}-2}}^{\ell}
+\left\|\left(\nabla a_{0},u_{0}\right)\right\|_{\dot{B}_{p,1}^{\frac
{d}{p}-1}}^{h}\leq c,
\end{equation*}
then the Cauchy problem \eqref{Eq:1.1}-\eqref{Eq:1.2} admits a unique global-in-time solution $(\varrho,u)$ with $\varrho=\varrho_{\infty}+a$ and $(a,u)$ in the space $X_{p}$ defined by:
\begin{eqnarray*}
&&\hspace{-6mm} a^{\ell} \in \widetilde{\mathcal{C}_{b}}\left(\mathbb{R_{+}};\dot{B}_{2,1}^{\frac {d}{2}-2}\right)\cap L^{1}\left(\mathbb{R_{+}};\dot{B}_{2,1}^{\frac {d}{2}}\right), \ \ u^{\ell} \in \widetilde{\mathcal{C}_{b}}\left(\mathbb{R_{+}};\dot{B}_{2,1}^{\frac {d}{2}-1}\right)\cap L^{1}\left(\mathbb{R_{+}};\dot{B}_{2,1}^{\frac {d}{2}+1}\right), \\
&&\hspace{-6mm} a^{h}\in \widetilde{\mathcal{C}_{b}}\left(\mathbb{R_{+}};\dot{B}_{p,1}^{\frac{d}{p}}\right)\cap L^{1}\left(\mathbb{R_{+}};\dot{B}_{p,1}^{\frac {d}{p}}\right),\ \ \ u^{h}\in \widetilde{\mathcal{C}_{b}}\left(\mathbb{R_{+}};\dot{B}_{p,1}^{\frac {d}{p}-1}\right)\cap L^{1}\left(\mathbb{R_{+}};\dot{B}_{p,1}^{\frac {d}{p}+1}\right).
\end{eqnarray*}
Furthermore, we get for some constant $C=C\left(p,d,\mu,\lambda,P,\varrho_{\infty}\right)$,
\begin{equation*}
\mathcal{E}_{p}(t)\leq C\mathcal{E}_{p,0},
\end{equation*}
for any $t>0$, where
\begin{eqnarray}
\mathcal{E}_{p}(t)&\triangleq&\left\|a\right\|_{\widetilde{L}_{t}^{\infty} (\dot{B}_{2,1}^{\frac {d}{2}-2})}^{\ell}
+\left\|u\right\|_{\widetilde{L}_{t}^{\infty} (\dot{B}_{2,1}^{\frac {d}{2}-1})}^{\ell}
+\left\|\left(a,\nabla u\right)\right\|_{L_{t}^{1} (\dot{B}_{2,1}^{\frac {d}{2}})}^{\ell} \nonumber \\
\label{Eq:1.6}
&&+\left\|\left(\nabla a,u\right)\right\|_{\widetilde{L}_{t}^{\infty}(\dot{B}_{p,1}^{\frac {d}{p}-1})}^{h}
+\left\|\left(a,\nabla u\right)\right\|_{L_{t}^{1}(\dot{B}_{p,1}^{\frac {d}{p}})}^{h}.
\end{eqnarray}
\end{thm}
The next step is to exhibit the large time asymptotic description of the constructed solution. Inspired by the dissipative analysis in \cite{CD2,LMZ},
we intend to establish the optimal time-decay estimates to \eqref{Eq:1.1}-\eqref{Eq:1.2}. It is convenient to rewrite \eqref{Eq:1.1} as the nonlinear perturbation form of $(\varrho_{\infty},0)$, looking at the nonlinearities as source terms. To
simplify the statement of our main result, we assume that $\varrho_{\infty}=1$ and $P'(1)=1$. Consequently,
in terms of the new variables $(a,u)$, System \eqref{Eq:1.1} becomes
\begin{equation} \label{Eq:1.7}
\left\{
\begin{array}{l}
\partial _{t}a+\div u=f, \\
\partial _{t}u-\mathcal{A}u+\nabla a+\nabla(-\Delta)^{-1} a=g,
\end{array}
\right.
\end{equation}
where
\begin{eqnarray*}
&&f \triangleq -\div \left(au\right), \\
&&g \triangleq -u\cdot \nabla u-I(a)\mathcal{A}u-k(a)\nabla a+\frac {1}{1+a}\mathrm{div}\left(2\widetilde{\mu }(a)\,D(u)+\widetilde{\lambda }(a) \div u \,\Id \right)
\end{eqnarray*}
with
\begin{center}
$\mathcal{A}\triangleq\mu _{\infty} \Delta +\left(\lambda _{\infty} +\mu _{\infty}\right)
\nabla \div$ such that $2\mu_{\infty}+\lambda _{\infty}=1$ and
$\mu_{\infty} >0$ \\
($\mu _{\infty}\triangleq\mu (1)$ and $\lambda _{\infty} \triangleq\lambda (1)$), \ \ $I(a) \triangleq \frac {a}{1+a}$, \ \
$k(a) \triangleq \frac{P^{\prime}(1+a)}{1+a}-1$, \\
$\widetilde{\mu }(a)\triangleq \mu (1+a)-\mu (1)$, \ \ $\widetilde{\lambda}(a)\triangleq \lambda(1+a)-\lambda (1)$.
\end{center}

The main result is stated as follows.
\begin{thm}\label{Thm1.2}
Let $d\geq2$ and $p$ satisfy \eqref{Eq:1.5}.  Denote by $(\varrho,u)$ the global solution constructed in Theorem \ref{Thm1.1}.
Let the real number $s_{1}$ fulfill
\begin{equation} \label{Eq:1.8}
1-\frac{d}{2}<s_{1}\leq s_{0} \ \ \hbox{with} \ \ s_{0}=\frac{2d}{p}-\frac{d}{2}.
\end{equation}
There exists a positive constant $c=c\left( p,d,\mu ,\lambda,P\right)$ such that if
\begin{equation} \label{Eq:1.9}
\mathcal{D}_{p,0}\triangleq \|a_{0}\|^{\ell}_{\dot{B}_{2,\infty}^{-s_{1}-1}}+\|u_{0}\|^{\ell}_{\dot{B}_{2,\infty}^{-s_{1}}}\leq c,
\end{equation}
then we have for all $t\geq 0$,
\begin{equation}\label{Eq:1.10}
\mathcal{D}_{p}(t)\lesssim \left(\mathcal{D}_{p,0}+\left\|\left(\nabla a_{0},u_{0}\right)\right\|^{h}_{\dot{B}_{p,1}^{\frac {d}{p}-1}}\right),
\end{equation}
where the functional $\mathcal{D}_{p}(t)$ is defined by
\begin{eqnarray}\label{Eq:1.11}
\mathcal{D}_{p}(t)&\triangleq& \sup_{s\in[\varepsilon-s_{1},\frac{d}{2}+1]}
\left\|\langle\tau\rangle^{\frac{s_{1}+s}{2}}(\tilde{a},u)\right\|^{\ell}_{L^{\infty}_{t}(\dot{B}_{2,1}^{s})}
+\left\|\langle\tau\rangle^{\alpha}\left(\nabla a,u\right)\right\|^{h}_{\tilde{L}^{\infty}_{t}(\dot{B}_{p,1}^{\frac {d}{p}-1})}\nonumber\\
&&+\left\|\tau^{\alpha} \nabla u \right\|^{h}_{\tilde{L}^{\infty}_{t}(\dot{B}_{p,1}^{\frac {d}{p}})}
\end{eqnarray}
with $\tilde{a}=\Lambda^{-1}a$ and $\alpha \triangleq s_{1}+\frac{d}{2}+\frac{1}{2}-\varepsilon$ for some sufficiently small $\varepsilon>0$.
\end{thm}

\begin{rem}\label{Rem1.1}
Condition \eqref{Eq:1.5} allows to the case $p>d$ for which the
regularity exponent $d/p-1$ of velocity field $u$ becomes negative in physical dimensions $d=2,3$, Theorem \ref{Thm1.2} thus holds for small data but
large highly oscillating initial velocity. Although similar results were available in \cite{CD2,BWY}, there are apparent innovative ingredients in Theorem \ref{Thm1.2}. More precisely, the work \cite{CD2} handle the case of $s_1=s_0$ and $p=2$ for simplicity, and \cite{BWY} only consider the non oscillation case ($2\leq p \leq d$) with the same endpoint regularity  $s_1=s_0$. In addition, the corresponding result of \cite{BWY} cannot cover the 2D case. In this paper,
we try to establish the time-weighted inequality in more general $L^p$ framework including cases $2\leq p \leq d$ and $p>d$. More important, the regularity assumption of low frequencies (\eqref{Eq:1.8}-\eqref{Eq:1.9}) is new, which enables us to enjoy larger freedom on the choice of $s_1$.

On the other hand, we emphasize that the decay exponents $\alpha=s_{1}+\frac{d}{2}+\frac{1}{2}-\varepsilon$ in the second and third terms of functional $\mathcal{D}_{p}(t)$ are optimal. Actually, it cannot be more than $s_{1}+\frac{d}{2}+\frac{1}{2}$. See for example,
$$
\|u^{\ell}\cdot \nabla u^{\ell}\|^{h}_{\dot{B}^{\frac{d}{p}-1}_{p,1}}\lesssim \|u^{\ell}\cdot \nabla u^{\ell}\|^{h}_{\dot{B}^{\frac{d}{p}}_{p,1}}
\lesssim \|u^{\ell}\|_{\dot{B}^{\frac{d}{p}}_{p,1}}\|\nabla u^{\ell}\|_{\dot{B}^{\frac{d}{p}}_{p,1}}.
$$
In particular, when $s_1=s_0$, the value of $\alpha$ becomes $\frac{2d}{p}+\frac 12-\varepsilon$ in \eqref{Eq:1.11}. Clearly, such optimal exponent
has not yet been observed in \cite{CD2,BWY}.
\end{rem}

\begin{rem}\label{Rem1.2}
Due to the dissipative effect arising from the Poisson potential, it follows from \eqref{Eq:1.10} that
the decay of density is \textit{faster} at the half rate than that of velocity. This is a totally different ingredient
in comparison with compressible Navier-Stokes equations (see for example \cite{DX}). As a matter of fact, we improve the analysis of
\cite{DX} such that the optimality of the regularity and decay exponents are available in the definition of $\mathcal{D}_{p}(t)$.
\end{rem}

As a consequence of Theorem \ref{Thm1.2}, we have the decay of the $L^p$ norm (the slightly stronger $\dot{B}^{0}_{p,1}$ norm in fact) of solutions.
\begin{cor}\label{Cor1.1}
The solution $(\varrho,u)$ constructed in Theorem \ref{Thm1.2} fulfills
$$\left\|\Lambda^{s}\left(\varrho-1\right)\right\|_{L^{p}}\lesssim  (\mathcal{D}_{p,0}+\left\|\left(\nabla a_{0},u_{0}\right)\right\|^{h}_{\dot{B}^{\frac{d}{p}-1}_{p,1}})\langle t\rangle^{-\frac{s_{1}+s+1}{2}} \ \hbox{if} \ -s_{1}-1<s\leq\frac{d}{p},$$
$$\left\|\Lambda^{s}u\right\|_{L^{p}}\lesssim (\mathcal{D}_{p,0}+\left\|\left(\nabla a_{0},u_{0}\right)\right\|^{h}_{\dot{B}^{\frac{d}{p}-1}_{p,1}})\langle t\rangle^{-\frac{s_{1}+s}{2}} \ \hbox{if} \ -s_{1}<s\leq\frac{d}{p}+1,$$
where the pseudo-differential operator $\Lambda^{\ell}$ is defined by $\Lambda^{\ell}f\triangleq\mathcal{F}^{-1}\left(|\xi|^{\ell}\mathcal{F}f\right)$.
\end{cor}

Moreover, one has more decay estimates of $L^{q}$-$L^{r}$ type.
\begin{cor}\label{Cor1.2}
Let the assumptions of Theorem \ref{Thm1.2} be satisfied with $p=2$. Then the corresponding solution $\varrho$ fulfills
$$\|\Lambda^{l}\left(\varrho-1\right)\|_{L^{r}}\lesssim (\mathcal{D}_{2,0}+\left\|\left(\nabla a_{0},u_{0}\right)\right\|^{h}_{\dot{B}^{\frac{d}{2}-1}_{2,1}})
\langle t\rangle^{-\frac{s_{1}}{2}-\frac{d}{2}(\frac{1}{2}-\frac{1}{r})-\frac{l+1}{2}}$$
for  $2\leq r\leq\infty$ and $l\in\mathbb{R}$ satisfying $-s_{1}-1<l+d\left(\frac{1}{2}-\frac{1}{r}\right)\leq\frac{d}{2}$, and $u$ fulfills
$$\|\Lambda^{k}u\|_{L^{r}}\lesssim (\mathcal{D}_{2,0}+\left\|\left(\nabla a_{0},u_{0}\right)\right\|^{h}_{\dot{B}^{\frac{d}{2}-1}_{2,1}})
\langle t\rangle^{-\frac{s_{1}}{2}-\frac{d}{2}(\frac{1}{2}-\frac{1}{r})-\frac{k}{2}}$$
for  $2\leq r\leq\infty$ and $k\in\mathbb{R}$ satisfying $-s_{1}<k+d\left(\frac{1}{2}-\frac{1}{r}\right)\leq\frac{d}{2}+1$.
\end{cor}

\begin{rem}\label{Rem1.3}
In contrast with previous efforts \cite{LMZ,LZ,WYJ,WW1,WW2} and references therein, those optimal decay estimates of the solution and its derivative are established in Besov spaces with the minimal regularity. The derivative index allows to take value in some interval rather than nonnegative integers only, and thus our results are optimal. In particular, compared to \cite{DX}, the choice of derivative indices of velocity is somewhat relaxed thanks to the optimal exponents for high frequencies.
\end{rem}

The proof of Theorem \ref{Thm1.2} is separated into three steps according to three terms in the functional $\mathcal{D}_{p}(t)$.
Although our proof follows from a similar fashion of the joint work \cite{DX},  we are interested in getting a \textit{sharp} time-weighted decay inequality in more general assumption of low frequencies.

Due to the coupled Poisson potential, there is a nonlocal term $\nabla(-\Delta)^{-1}a$ (that is equivalent to $\Lambda^{-1}a$) available in the velocity equation, which forces us to study the corresponding asymptotic behavior. As in \cite{CD1}, it is convenient to consider \eqref{Eq:3.4} in terms of $(\tilde{a},u)$ in the low-frequency regime, see Section 3. By the spectral analysis, one can conclude that the Green function for the linearized form of \eqref{Eq:3.4} behaves like a heat equation at low frequency. Consequently, it is possible to adapt the standard Duhamel principle handling those nonlinear terms in the right hand side of \eqref{Eq:3.4}. Owing to \eqref{Eq:1.9} in which $s_{1}$ belongs to the whole range $(1-\frac{d}{2},s_{0}]$, the low-frequency analysis (in the first step) is more complicated than \cite{DX}. More precisely, with the aid of low and high frequency decomposition, we separate the nonlinear term $\left(\Lambda^{-1}f,g\right)$ into $\left(\Lambda^{-1}f^{\ell},g^{\ell}\right)$ and $\left(\Lambda^{-1}f^{h},g^{h}\right)$ (see the context below).
To bound $\left(\Lambda^{-1}f^{\ell},g^{\ell}\right)$, we develop non classical product estimates in Besov spaces, see \eqref{Eq:3.9}-\eqref{Eq:3.10}.
On the other hand, bounding $\left(\Lambda^{-1}f^{h},g^{h}\right)$, we proceed differently the decay analysis depending on whether $2\leq p\leq d$ (non oscillation) or $p>d$ (oscillation). The former case lies in a new Besov product estimates (see \eqref{Eq:3.18}), while the later case (that is relevant in physical dimension $d=2,3$) depends on non-classical product estimates in Proposition \ref{Prop2.3}.

In the high-frequency regime, the nonlocal term $\nabla(-\Delta)^{-1}a$  is no longer effective and the Green function behaves the same as that of compressible Navier-Stoke equations.
In contrast with \cite{DX},  the main objective of the second and third steps is to track the optimal decay exponents with the general assumption \eqref{Eq:1.9}. For that end, we introduce the \textit{effective velocity} (which was initiated by Hoff \cite{HD} and well developed by Haspot \cite{HB})
\begin{equation}\label{Eq:1.12}
w\triangleq \nabla \left(-\Delta \right)^{-1}\left( a-\mathrm{div}\,u\right)
\end{equation}
in the second step and eliminate the technical difficulty that there is a loss of one derivative of density in the convention term. The last step is devoted to
establish gain of regularity and decay altogether for the high frequency of velocity. The analysis strongly relies on the parabolic maximal regularity for the Lam\'{e} semi-group (that are the same as for the heat semi-group, see Remark \ref{Rem2.2}). In order to highlight the new observation on the decay exponents, we present the proof in details, which may be of interest for the further works.

The rest of the paper unfolds as follows: In section 2, we briefly recall Littlewood-Paley decomposition, Besov spaces  and useful analysis tools.
Section 3 is devoted to those proofs of Theorem \ref{Thm1.2} and Corollaries \ref{Cor1.1}-\ref{Cor1.2}.

\section{Preliminary}\setcounter{equation}{0}
Throughout the paper, $C>0$ stands for a generic ``constant''. For brevity, we write
$f\lesssim g$ instead of $f\leq Cg$. The notation $f\approx g$ means that $%
f\lesssim g$ and $g\lesssim f$. For any Banach space $X$ and $f,g\in X$, we agree that
$\left\|(f,g)\right\| _{X}\triangleq \left\|f\right\| _{X}+\left\|g\right\|_{X}$. For
all $T>0$ and $\theta \in[1,+\infty]$, we denote by
$L_{T}^{\theta}(X) \triangleq L^{\theta}\left([0,T];X\right)$ the set of measurable functions $f:[0,T]\rightarrow X$ such that $t\mapsto\left\|f(t)\right\|_{X}$ is in $L^{\theta}(0,T)$.

To make the paper self-contained, we briefly recall Littlewood-Paley decomposition, Besov spaces and analysis tools.
The reader is referred to Chap. 2 and Chap. 3 of \cite{BCD} for more details. Let's begin with the homogeneous Littlewood-Paley decomposition. For that purpose, we fix some smooth
radial non increasing function $\chi $ with $\mathrm{Supp}\,\chi \subset
B\left(0,\frac {4}{3}\right)$ and $\chi \equiv 1$ on $B\left(0,\frac
{3}{4}\right)$, then set $\varphi (\xi) =\chi (\xi/2)-\chi (\xi)$
so that
$$
\sum_{j\in \mathbb{Z}}\varphi ( 2^{-j}\cdot ) =1\ \ \hbox{in}\ \
\mathbb{R}^{d}\setminus \{ 0\} \ \ \hbox{and}\ \ \mathrm{Supp}\,\varphi \subset \left\{ \xi \in \mathbb{R}^{d}:3/4\leq |\xi|\leq 8/3\right\} .
$$
The homogeneous dyadic blocks $\dot{\Delta}_{j}$ are defined by
$$
\dot{\Delta}_{j}f\triangleq \varphi (2^{-j}D)f=\mathcal{F}^{-1}\left(\varphi
(2^{-j}\cdot )\mathcal{F}f\right)=2^{jd}h(2^{j}\cdot )\star f\ \ \hbox{with}\ \
h\triangleq \mathcal{F}^{-1}\varphi .
$$
Formally, we have the homogeneous decomposition as follows
\begin{equation} \label{Eq:2.1}
f=\sum_{j\in \mathbb{Z}}\dot{\Delta}_{j}f,
\end{equation}
for any tempered distribution $f\in S^{\prime }(\mathbb{R}^{d})$. As it holds only modulo polynomials,
it is convenient to consider the subspace of those tempered distributions $f$ such that
\begin{equation}\label{Eq:2.2}
\lim_{j\rightarrow -\infty }\left\| \dot{S}_{j}f\right\| _{L^{\infty} }=0,
\end{equation}
where $\dot{S}_{j}f$ stands for the low frequency cut-off defined by $\dot{S}_{j}f\triangleq\chi (2^{-j}D)f$. Indeed,
if \eqref{Eq:2.2} is fulfilled, then \eqref{Eq:2.1} holds in $S'(\mathbb{R}^{d})$. For convenience, we denote by $S'_{0}(\mathbb{R}^{d})$ the subspace of tempered distributions satisfying \eqref{Eq:2.2}.

The homogeneous Besov space is defined in terms of above Littlewood-Paley decomposition.
\begin{defn}\label{Defn2.1}
For $\sigma\in \mathbb{R}$ and $1\leq p,r\leq\infty,$ the homogeneous
Besov spaces $\dot{B}^{\sigma}_{p,r}$ is defined by
$$\dot{B}^{\sigma}_{p,r}\triangleq\left\{f\in S'_{0}:\left\|f\right\|_{\dot{B}^{\sigma}_{p,r}}<+\infty\right\},$$
where
\begin{equation}\label{Eq:2.3}
\left\|f\right\|_{\dot B^{\sigma}_{p,r}}\triangleq\left\|\left(2^{j\sigma}\left\|\ddj  f\right\|_{L^p}\right)\right\|_{\ell^{r}(\Z)}.
\end{equation}
\end{defn}

We often use the following classical properties (see \cite{BCD}):

$\bullet$ \ \emph{Scaling invariance:} For any $\sigma\in \mathbb{R}$ and $(p,r)\in
[1,\infty ]^{2}$, there exists a constant $C=C(\sigma,p,r,d)$ such that for all $\lambda >0$ and $f\in \dot{B}_{p,r}^{\sigma}$, we have
$$
C^{-1}\lambda ^{\sigma-\frac {d}{p}}\left\|f\right\|_{\dot{B}_{p,r}^{\sigma}}
\leq \left\|f(\lambda \cdot)\right\|_{\dot{B}_{p,r}^{\sigma}}\leq C\lambda ^{\sigma-\frac {d}{p}}\left\|f\right\|_{\dot{B}_{p,r}^{\sigma}}.
$$

$\bullet$ \ \emph{Completeness:} $\dot{B}^{\sigma}_{p,r}$ is a Banach space whenever $%
\sigma<\frac{d}{p}$ or $\sigma\leq \frac{d}{p}$ and $r=1$.

$\bullet$ \ \emph{Interpolation:} The following inequality is satisfied for $1\leq p,r_{1},r_{2}, r\leq \infty, \sigma_{1}\neq \sigma_{2}$ and $\theta \in (0,1)$:
$$\left\|f\right\|_{\dot{B}_{p,r}^{\theta \sigma_{1}+(1-\theta )\sigma_{2}}}\lesssim \left\|f\right\| _{\dot{B}_{p,r_{1}}^{\sigma_{1}}}^{\theta} \left\|f\right\|_{\dot{B}_{p,r_2}^{\sigma_{2}}}^{1-\theta }$$
with $\frac{1}{r}=\frac{\theta}{r_{1}}+\frac{1-\theta}{r_{2}}$.

$\bullet$ \ \emph{Action of Fourier multipliers:} If $F$ is a smooth homogeneous of
degree $m$ function on $\mathbb{R}^{d}\backslash \{0\}$ then
$$F(D):\dot{B}_{p,r}^{\sigma}\rightarrow \dot{B}_{p,r}^{\sigma-m}.$$

The embedding properties will be used several times throughout the paper.
\begin{prop} \label{Prop2.1} (Embedding for Besov spaces on $\mathbb{R}^{d}$)
\begin{itemize}
\item For any $p\in[1,\infty]$ we have the  continuous embedding
$\dot {B}^{0}_{p,1}\hookrightarrow L^{p}\hookrightarrow \dot {B}^{0}_{p,\infty}.$
\item If $\sigma\in\R$, $1\leq p_{1}\leq p_{2}\leq\infty$ and $1\leq r_{1}\leq r_{2}\leq\infty,$
then $\dot {B}^{\sigma}_{p_1,r_1}\hookrightarrow
\dot {B}^{\sigma-d(\frac{1}{p_{1}}-\frac{1}{p_{2}})}_{p_{2},r_{2}}$.
\item The space  $\dot {B}^{\frac {d}{p}}_{p,1}$ is continuously embedded in the set  of
bounded  continuous functions \emph{(}going to zero at infinity if, additionally, $p<\infty$\emph{)}.
\end{itemize}
\end{prop}
The product estimate in Besov spaces plays a fundamental role in bounding bilinear terms of \eqref{Eq:1.7}  (see for example, \cite{BCD,DX}).
\begin{prop}\label{Prop2.2}
Let $\sigma>0$ and $1\leq p,r\leq\infty$. Then $\dot{B}^{\sigma}_{p,r}\cap L^{\infty}$ is an algebra and
$$
\left\|fg\right\|_{\dot{B}^{\sigma}_{p,r}}\lesssim \left\|f\right\|_{L^{\infty}}\left\|g\right\|_{\dot{B}^{\sigma}_{p,r}}+\left\|g\right\|_{L^{\infty}}\left\|f\right\|_{\dot{B}^{\sigma}_{p,r}}.
$$
Let the real numbers $\sigma_{1},$ $\sigma_{2},$ $p_1$  and $p_2$ fulfill
$$
\sigma_{1}+\sigma_{2}>0,\quad \sigma_{1}\leq\frac {d}{p_{1}},\quad\sigma_{2}\leq\frac {d}{p_{2}},\quad
\sigma_{1}\geq\sigma_{2},\quad\frac{1}{p_{1}}+\frac{1}{p_{2}}\leq1.
$$
Then we have
$$\left\|fg\right\|_{\dot{B}^{\sigma_{2}}_{q,1}}\lesssim \left\|f\right\|_{\dot{B}^{\sigma_{1}}_{p_{1},1}}\left\|g\right\|_{\dot{B}^{\sigma_{2}}_{p_{2},1}}\quad\hbox{with}\quad
\frac1{q}=\frac1{p_{1}}+\frac1{p_{2}}-\frac{\sigma_{1}}d\cdotp$$
Additionally, for exponents $\sigma>0$ and $1\leq p_{1},p_{2},q\leq\infty$ satisfying
$$\frac{d}{p_{1}}+\frac{d}{p_{2}}-d\leq \sigma \leq\min\left(\frac {d}{p_{1}},\frac {d}{p_{2}}\right)\quad\hbox{and}\quad \frac{1}{q}=\frac {1}{p_{1}}+\frac {1}{p_{2}}-\frac{\sigma}{d},$$
we have
$$\left\|fg\right\|_{\dot{B}^{-\sigma}_{q,\infty}}\lesssim\left\|f\right\|_{\dot{B}^{\sigma}_{p_{1},1}}\left\|g\right\|_{\dot{B}^{-\sigma}_{p_{2},\infty}}.$$
\end{prop}

To handle the case of $p>d$ in the proof of Theorem \ref{Thm1.2}, just resorting to Proposition \ref{Prop2.2} does not allow to get suitable bounds
for the low frequency part of some nonlinear terms, so we need to take advantage of the following non-classical product estimates (see \cite{DX}).
\begin{prop}\label{Prop2.3} Let $j_{0}\in\Z,$ and denote $z^{\ell}\triangleq\dot S_{j_{0}}z$, $z^{h}\triangleq z-z^{\ell}$ and, for any $s\in\R$,
$$
\left\|z\right\|_{\dot B^{\sigma}_{2,\infty}}^{\ell}\triangleq\sup_{j\leq j_{0}}2^{j\sigma} \left\|\ddj z\right\|_{L^2}.
$$
There exists a universal integer $N_{0}$ such that  for any $2\leq p\leq 4$ and $\sigma>0,$ we have
\begin{eqnarray}\label{Eq:2.4}
&&\left\|f g^{h}\right\|_{\dot {B}^{-s_{0}}_{2,\infty}}^{\ell}\leq C \left(\left\|f\right\|_{\dot {B}^{\sigma}_{p,1}}+\left\|\dot S_{j_{0}+N_{0}}f\right\|_{L^{{p}^{*}}}\right)\left\|g^{h}\right\|_{\dot{B}^{-\sigma}_{p,\infty}},\\
\label{Eq:2.5}
&&\left\|f^{h} g\right\|_{\dot {B}^{-s_{0}}_{2,\infty}}^{\ell}
\leq C \left(\left\|f^{h}\right\|_{\dot{B}^{\sigma}_{p,1}}+\left\|\dot{S}_{j_{0}+N_{0}}f^{h}\right\|_{L^{p^{*}}}\right)\left\|g\right\|_{\dot {B}^{-\sigma}_{p,\infty}}
\end{eqnarray}
with  $s_{0}\triangleq \frac{2d}{p}-\frac {d}{2}$ and $\frac1{p^{*}}\triangleq\frac{1}{2}-\frac{1}{p},$
and $C$ depending only on $j_{0}$, $d$ and $\sigma$.
\end{prop}
System \eqref{Eq:1.7} also involves compositions of functions (through $I(a)$, $k(a)$, $\tilde{\lambda}(a)$ and $\tilde{\mu}(a)$) that
are handled according to the following conclusion.
\begin{prop}\label{Prop2.4}
Let $F:\R\rightarrow\R$ be  smooth with $F(0)=0$.
For  all  $1\leq p,r\leq\infty$ and $\sigma>0$ we have
$F(f)\in \dot {B}^{\sigma}_{p,r}\cap L^{\infty}$  for  $f\in \dot {B}^{\sigma}_{p,r}\cap L^{\infty},$  and
$$\left\|F(f)\right\|_{\dot B^\sigma_{p,r}}\leq C\left\|f\right\|_{\dot B^\sigma_{p,r}}$$
with $C$ depending only on $\left\|f\right\|_{L^{\infty}}$, $F'$ \emph{(}and higher derivatives\emph{)}, $\sigma$, $p$ and $d$.

In the case $\sigma>-\min\left(\frac {d}{p},\frac {d}{p'}\right)$ then $f\in\dot {B}^{\sigma}_{p,r}\cap\dot {B}^{\frac {d}{p}}_{p,1}$
implies that $F(f)\in \dot {B}^{\sigma}_{p,r}\cap\dot {B}^{\frac {d}{p}}_{p,1}$, and we have
$$\left\|F(f)\right\|_{\dot B^{\sigma}_{p,r}}\leq C(1+\left\|f\right\|_{\dot {B}^{\frac {d}{p}}_{p,1}})\left\|f\right\|_{\dot {B}^{\sigma}_{p,r}}.$$
\end{prop}

In addition, we also recall the classical \emph{Bernstein inequality}:
\begin{equation}\label{Eq:2.6}
\left\|D^{k}f\right\|_{L^{b}}
\leq C^{1+k} \lambda^{k+d(\frac{1}{a}-\frac{1}{b})}\left\|f\right\|_{L^{a}}
\end{equation}
that holds for all function $f$ such that $\mathrm{Supp}\,\mathcal{F}f\subset\left\{\xi\in \R^{d}: \left|\xi\right|\leq R\lambda\right\}$ for some $R>0$
and $\lambda>0$, if $k\in\N$ and $1\leq a\leq b\leq\infty$.

More generally, if we assume $f$ to satisfy $\mathrm{Supp}\,\mathcal{F}f\subset \left\{\xi\in \R^{d}:
R_{1}\lambda\leq\left|\xi\right|\leq R_{2}\lambda\right\}$ for some $0<R_{1}<R_{2}$  and $\lambda>0$,
then for any smooth  homogeneous of degree $m$ function $A$ on $\R^d\setminus\{0\}$ and $1\leq a\leq\infty,$ we have
(see e.g. Lemma 2.2 in \cite{BCD}):
\begin{equation}\label{Eq:2.7}
\left\|A(D)f\right\|_{L^{a}}\approx\lambda^{m}\left\|f\right\|_{L^{a}}.
\end{equation}
An obvious  consequence of \eqref{Eq:2.6} and \eqref{Eq:2.7} is that
$\left\|D^{k}f\right\|_{\dot{B}^{s}_{p, r}}\thickapprox\left \|f\right\|_{\dot{B}^{s+k}_{p, r}}$ for all $k\in\N.$
Also, the following commutator estimate (see \cite{DX}) will be used in the second step of the proof of Theorem \ref{Thm1.2}.
\begin{prop}\label{Prop2.5}
Let $1\leq p,\,p_{1}\leq\infty$ and
\begin{equation*}
-\min\left(\frac{d}{p_{1}},\frac{d}{p'}\right)<\sigma\leq1+\min\left(\frac {d}{p},\frac{d}{p_{1}}\right).
\end{equation*}
There exists a constant $C>0$ depending only on $\sigma$ such that for all $j\in\Z$ and $\ell\in\left\{1,\cdots,d\right\}$, we have
\begin{equation*}
\left\|\left[v\cdot\nabla,\d_\ell\dot{\Delta}_{j}\right]a\right\|_{L^{p}}\leq
Cc_{j}2^{-j\left(\sigma-1\right)}\left\|\nabla v\right\|_{\dot{B}^{\frac{d}{p_{1}}}_{p_{1},1}}\left\|\nabla a\right\|_{\dot{B}^{\sigma-1}_{p,1}},
\end{equation*}
where the commutator
$\left[\cdot,\cdot\right]$ is defined by $\left[f,g\right]=fg-gf$ and $(c_{j})_{j\in\Z}$ denotes
a sequence such that $\left\|(c_{j})\right\|_{\ell^{1}}\leq 1$.
\end{prop}

On the other hand, a class of mixed space-time Besov spaces are also used, which was initiated by J.-Y. Chemin and N. Lerner \cite{CL} (see also \cite{CJY} for the particular case of Sobolev spaces).
\begin{defn}\label{Defn2.2}
 For $T>0, \sigma\in\mathbb{R}, 1\leq r,\theta\leq\infty$, the homogeneous Chemin-Lerner space $\widetilde{L}^{\theta}_{T}(\dot{B}^{\sigma}_{p,r})$
is defined by
$$\widetilde{L}^{\theta}_{T}(\dot{B}^{\sigma}_{p,r})\triangleq\left\{f\in L^{\theta}\left(0,T;S'_{0}\right):\left\|f\right\|_{\widetilde{L}^{\theta}_{T}(\dot{B}^{\sigma}_{p,r})}<+\infty\right\},$$
where
\begin{equation}\label{Eq:2.8}
\left\|f\right\|_{\widetilde{L}^{\theta}_{T}(\dot{B}^{\sigma}_{p,r})}\triangleq\left\|\left(2^{j\sigma}\left\|\ddj  f\right\|_{L^{\theta}_{T}(L^{p})}\right)\right\|_{\ell^{r}(\Z)}.
\end{equation}
\end{defn}
For notational simplicity, index $T$ is omitted if $T=+\infty $.
We agree with the notation
\begin{equation*}
\tilde{\mathcal{C}}_{b}(\mathbb{R_{+}};\dot{B}_{p,r}^{\sigma})\triangleq \left\{f \in
\mathcal{C}(\mathbb{R_{+}};\dot{B}_{p,r}^{\sigma})\ \hbox{s.t}\ \left\|f\right\| _{\tilde{L}^{\infty}(\dot{B}_{p,r}^{\sigma})}<+\infty \right\} .
\end{equation*}
The Chemin-Lerner space $\widetilde{L}^{\theta}_{T}(\dot{B}^{\sigma}_{p,r})$ may be linked with the standard spaces $L_{T}^{\theta} (\dot{B}_{p,r}^{\sigma})$ by means of Minkowski's
inequality.
\begin{rem}\label{Rem2.1}
It holds that
$$\left\|f\right\|_{\widetilde{L}^{\theta}_{T}(B^{\sigma}_{p,r})}\leq\left\|f\right\|_{L^{\theta}_{T}(B^{\sigma}_{p,r})}\,\,\,
\mbox{if} \,\, \, r\geq\theta;\ \ \ \
\left\|f\right\|_{\widetilde{L}^{\theta}_{T}(B^{\sigma}_{p,r})}\geq\left\|f\right\|_{L^{\theta}_{T}(B^{\sigma}_{p,r})}\,\,\,
\mbox{if}\,\,\, r\leq\theta.
$$
\end{rem}
Restricting the above norms \eqref{Eq:2.3} and \eqref{Eq:2.8} to the low or high
frequencies parts of distributions will be fundamental in our method. For instance, let us fix some integer $j_{0}$ (the value of which will
follow from the proof of the high-frequency estimates) and put\footnote{Note that for technical reasons, we need a small
overlap between low and high frequencies.}
$$\left\| f\right\| _{\dot{B}_{p,1}^{\sigma}}^{\ell} \triangleq \sum_{j\leq
j_{0}}2^{j\sigma}\left\| \dot{\Delta}_{j}f\right\|_{L^{p}} \ \mbox{and} \ \left\|f\right\|_{\dot{B}_{p,1}^{\sigma}}^{h}\triangleq \sum_{j\geq j_{0}-1}2^{j\sigma}\left\| \dot{\Delta}_{j}f\right\| _{L^{p}},$$
$$\left\|f\right\| _{\tilde{L}_{T}^{\infty} (\dot{B}_{p,1}^{\sigma})}^{\ell} \triangleq
\sum_{j\leq j_{0}}2^{j\sigma}\left\|\dot{\Delta}_{j}f\right\|_{L_{T}^{\infty} (L^{p})} \
\mbox{and} \ \left\|f\right\| _{\tilde{L}_{T}^{\infty} (\dot{B}_{p,1}^{\sigma})}^{h}\triangleq \sum_{j\geq j_{0}-1}2^{j\sigma}\left\| \dot{\Delta}_{j}f\right\|
_{L_{T}^{\infty} (L^{p})}.$$

Let us finally recall the parabolic regularity estimate for the
heat equation to end this section.
\begin{prop}\label{Prop2.6}
Let $\sigma\in \R$, $(p,r)\in \left[1,\infty\right]^{2}$ and $1\leq \rho_{2}\leq\rho_{1}\leq\infty$. Let $u$ satisfy
$$\left\{\begin{array}{lll}
\d_{t}u-\mu\Delta u=f,\\
u_{|t=0}=u_{0}.
\end{array}
\right.$$
Then for all $T>0$ the following a priori estimate is fulfilled:
\begin{equation}\label{Eq:2.9}\mu^{\frac1{\rho_1}}\left\|u\right\|_{\tilde L_{T}^{\rho_1}(\dot B^{\sigma+\frac{2}{\rho_1}}_{p,r})}\lesssim
\left\|u_{0}\right\|_{\dot {B}^{\sigma}_{p,r}}+\mu^{\frac{1}{\rho_{2}}-1}\left\|f\right\|_{\tilde L^{\rho_{2}}_{T}(\dot {B}^{\sigma-2+\frac{2}{\rho_{2}}}_{p,r})}.
\end{equation}
\end{prop}
\begin{rem} \label{Rem2.2}
The solutions to the following \emph{Lam\'e system}
\begin{equation}\label{Eq:2.10}
\left\{\begin{array}{lll}\d_tu-\mu\Delta u-\left(\lambda+\mu\right)\nabla \div u=f,\\
u_{|t=0}=u_{0},
\end{array}
\right.
\end{equation}
where $\lambda$ and $\mu$ are constant coefficients such that $\mu>0$ and $\lambda+2\mu>0,$
also fulfill \eqref{Eq:2.9} (up to the dependence w.r.t. the viscosity). Indeed, if we denote by $\mathcal{P}\triangleq\mathrm{Id}-\nabla(-\Delta)^{-1}\mathrm{div}$
and $\mathcal{Q}\triangleq\mathrm{Id}-\mathcal{P}$ the orthogonal projectors over divergence-free
and potential vector fields, then we see both $\mathcal{P}u$ and $\mathcal{Q} u$ satisfy the heat equation,
as it can easily be observed by applying $\mathcal{P}$ and $\mathcal{Q}$ to \eqref{Eq:2.10}.
\end{rem}

\section{The proof of Time-decay estimates} \setcounter{equation}{0}
In this section, our central task is to prove Theorem \ref{Thm1.2} taking for granted the global-in-time existence result of Theorem \ref{Thm1.1}.
The proof is divided into three steps, according to the three terms of the time-weighted functional $\mathcal{D}_{p}(t)$ (see \eqref{Eq:1.11}). In what follows, we shall use frequently elementary inequalities for $0\leq \sigma_{1}\leq \sigma_{2}$ with $\sigma_{2}>1$:
\begin{equation}\label{Eq:3.1}
\int_{0}^{t}\langle t-\tau\rangle^{-\sigma_{1}}\tau^{-\theta}\langle\tau\rangle^{\theta-\sigma_{2}}d\tau\lesssim\langle t\rangle^{-\sigma_{1}} \ \ \hbox{if} \ \ 0\leq\theta<1.
\end{equation}
Let us keep in mind that the global solution $\left(a,u\right)$ given by Theorem \ref{Thm1.1} satisfies
\begin{equation}\label{Eq:3.2}
\left\|a\right\|_{\tilde{L}^{\infty}_{t}(\dot{B}^{\frac{d}{p}}_{p,1})}\leq c \ll 1 \ \ \hbox{for all} \ \ t\geq0.
\end{equation}
\subsection{First Step: Bounds for the Low Frequencies}
Set $\omega=\Lambda^{-1}\mathrm{div}\,u$, $\Omega=\Lambda^{-1}\mathrm{curl}\,u$ with $\Lambda^{s}z\triangleq\mathcal{F}^{-1}\left(\left|\xi\right|^{s}\mathcal{F}z\right)$ ($s\in\mathbb{R}$). Then system \eqref{Eq:1.7} becomes
\begin{equation*}
\left\{
\begin{array}{l}
\partial _{t}a +\Lambda \omega=f, \\
\partial _{t}\omega-\Delta \omega-\Lambda a-\Lambda^{-1}a=\Lambda^{-1}\mathrm{div}\,g, \\
\partial _{t}\Omega-\mu_{\infty}\Delta \Omega=\Lambda^{-1}\mathrm{curl}\,g, \\
u=-\Lambda^{-1}\nabla\omega+\Lambda^{-1}\mathrm{div}\,\Omega.
\end{array}
\right.
\end{equation*}
Observe that the incompressible component $\Omega$ satisfies a mere heat equation.
As pointed out in Introduction, at low frequencies, it
is natural to consider the following system
\begin{equation}\label{Eq:3.4}
\left\{
\begin{array}{l}
\partial _{t}\tilde{a} +\omega=\Lambda^{-1}f, \\
\partial _{t}\omega-\Delta \omega-(\Lambda^{2}+1)\tilde{a}=\Lambda^{-1}\mathrm{div}\,g,
\end{array}
\right.
\end{equation}
with $\tilde{a}=\Lambda^{-1}a$.

Let $(\tilde{E}(t))_{t\geq0}$ be the semi-group associated with the left-hand side of \eqref{Eq:3.4}.
From an explicit computation of the action of $\tilde{E}(t)$ in Fourier variables (see for example \cite{CD2}), one can conclude that for all $j_{0}\in\mathbb{Z}$, there exist positive constants $c_{0}$ and $C$ depending only on $j_{0}$ such that
\begin{equation}\label{Eq:3.5}
\left|\mathcal{F}\left(\tilde{E}(t)\tilde{U}\right)(\xi)\right|\leq Ce^{-c_{0}\left|\xi\right|^{2}t}\left|\mathcal{F}\tilde{U}(\xi)\right| \ \ \hbox{for all} \ \ |\xi|\leq 2^{j_{0}},
\end{equation}
where $\tilde{U}\triangleq \left(\tilde{a},\omega\right)$. Then it follows from \cite{DX} that
\begin{equation*}
\sup_{t\geq 0}\langle t\rangle^{\frac{s_{1}+s}{2}}\left\|\tilde{E}(t)\tilde{U}_{0}\right\|^{\ell}_{\dot{B}^{s}_{2,1}}\lesssim \left\|\tilde{U}_{0}\right\|^{\ell}_{\dot{B}^{-s_{1}}_{2,\infty}}, \ \ \hbox{if} \ \ s_{1}+s>0,
\end{equation*}
where we denote $\langle t\rangle\triangleq\sqrt{1+t^{2}}$ and $\tilde{U}_{0}\triangleq \left(\tilde{a}_{0},\omega_{0}\right)$. Of course, the incompressible component $\Omega$ also satisfies \eqref{Eq:3.5}.
Consequently, we end up with
\begin{equation}\label{Eq:3.6}
\left\|\left(\tilde{a},u\right)\right\|^{\ell}_{\dot{B}^{s}_{2,1}}\lesssim
\langle t\rangle^{-\frac{s_{1}+s}{2}}\left\|\left(\tilde{a}_{0},u_{0}\right)\right\|^{\ell}_{\dot{B}^{-s_{1}}_{2,\infty}}
+\int^{t}_{0}\langle t-\tau\rangle^{-\frac{s_{1}+s}{2}}\left\|\left(\Lambda^{-1}f,g\right)\right\|^{\ell}_{\dot{B}^{-s_{1}}_{2,\infty}}d\tau.
\end{equation}

Bounding the time-weighted integral on the right side of \eqref{Eq:3.6} is included in the following proposition.
\begin{prop}\label{Prop3.1}
If $p$ fulfills \eqref{Eq:1.5}, then it holds that for all $t\geq0$,
\begin{equation}\label{Eq:3.7}
\int^{t}_{0}\langle t-\tau\rangle^{-\frac{s_{1}+s}{2}}\left\|\left(\Lambda^{-1}f,g\right)\right\|^{\ell}_{\dot{B}^{-s_{1}}_{2,\infty}}d\tau
\lesssim \langle t\rangle^{-\frac{s_{1}+s}{2}}\left(\mathcal{D}^{2}_{p}(t)+\mathcal{E}^{2}_{p}(t)\right),
\end{equation}
provided that $-s_{1}<s\leq\frac{d}{2}+1$, where $\mathcal{E}_{p}(t)$ and $\mathcal{D}_{p}(t)$ have been defined in \eqref{Eq:1.6} and \eqref{Eq:1.11}, respectively.
\end{prop}

For clarity, we split the proof of Proposition \ref{Prop3.1} into several lemmas. It is suitable to decompose $\Lambda^{-1}f$ and $g$ in terms
of low-frequency and high-frequency as follows:
$$\Lambda^{-1}f=\Lambda^{-1}f^{\ell}+\Lambda^{-1}f^{h}$$
with
\begin{equation*}
\Lambda^{-1}f^{\ell}\triangleq-\Lambda^{-1}\mathrm{div}\,(au^{\ell}), \ \ \ \ \ \Lambda^{-1}f^{h}\triangleq-\Lambda^{-1}\mathrm{div}\,(au^{h})
\end{equation*}
and
$$g= g^{\ell}+g^{h}$$
with
\begin{eqnarray*}
&&g^{\ell}\triangleq-u\cdot \nabla u^{\ell}-k( a) \nabla a^{\ell}+g_{3}(a,u^{\ell})+g_{4}(a,u^{\ell}),\\
&&g^{h}\triangleq-u\cdot \nabla u^{h}-k(a) \nabla a^{h}+g_{3}(a,u^{h})+g_{4}(a,u^{h}),
\end{eqnarray*}
where
\begin{eqnarray*}
g_{3}(a,v)&=&\frac {1}{1+a}\left( 2\widetilde{\mu }(a)\,\mathrm{div}\,D(v)+
\widetilde{\lambda}(a)\,\nabla \mathrm{div}\,v\right) -I(a)\mathcal{A}v, \\
g_{4}(a,v)&=&\frac {1}{1+a}\left( 2\widetilde{\mu}'(a)\,\mathrm{div}\,D(v)\cdot\nabla a+\widetilde{\lambda}'(a)\,\mathrm{div}\,v\, \nabla
a\right)
\end{eqnarray*}
and
$$z^{\ell}\triangleq\sum_{j< j_{0}} \dot{\Delta}_{j}z,\ \  \ \ z^{h}\triangleq z-z^{\ell} \ \ \hbox{for} \ \ z=a,u.$$

\begin{lem}\label{Lem3.1} If $p$ satisfies \eqref{Eq:1.5}, then it holds that for all $t\geq0$,
\begin{equation}\label{Eq:3.8}
\int^{t}_{0}\langle t-\tau\rangle^{-\frac{s_{1}+s}{2}}\left\|\left(\Lambda^{-1}f^{\ell},g^{\ell}\right)\right\|^{\ell}_{\dot{B}^{-s_{1}}_{2,\infty}}d\tau
\lesssim \langle t\rangle^{-\frac{s_{1}+s}{2}}\left(\mathcal{D}^{2}_{p}(t)+\mathcal{E}^{2}_{p}(t)\right).
\end{equation}
\end{lem}
\begin{proof}
To deal with five terms of $\Lambda^{-1}f$ and $g$ with $a^{\ell}$ or $u^{\ell}$, we present two non classical product inequalities:
\begin{eqnarray}\label{Eq:3.9}
&&\left\|FG\right\|_{\dot{B}^{-s_{1}}_{2,\infty}}
\lesssim \left\|F\right\|_{\dot{B}^{\frac{d}{p}}_{p,1}}\left\|G\right\|_{\dot{B}^{-s_{1}}_{2,1}},\\
\label{Eq:3.10}
&&\left\|FG\right\|_{\dot{B}^{\frac{d}{p}-\frac{d}{2}-s_{1}}_{2,\infty}}
\lesssim \left\|F\right\|_{\dot{B}^{\frac{d}{p}-\frac{d}{2}-s_{1}}_{p,1}}\left\|G\right\|_{\dot{B}^{\frac{d}{p}}_{2,1}},
\end{eqnarray}
which are the direct consequences of Proposition \ref{Prop2.2} with the assumptions \eqref{Eq:1.5} and \eqref{Eq:1.8}.
The details are left to the interested reader.

On the other hand, we shall use frequently that, owing to the embedding properties, the definition of $\mathcal{D}_{p}(t)$ and the fact that $\alpha\triangleq s_{1}+\frac{d}{2}+\frac{1}{2}-\varepsilon\geq\frac{s_{1}}{2}+\frac{d}{4}+\frac{1}{2}$ for small enough $\varepsilon>0$,
\begin{equation}\label{Eq:3.11}
\sup_{\tau \in [0,t]}\langle\tau\rangle^{\frac{s_{1}}{2}+\frac{d}{4}+\frac{1}{2}} \left\|a(\tau)\right\|_{\dot{B}^{\frac{d}{p}}_{p,1}}
+\sup_{\tau \in [0,t]}\langle\tau\rangle^{\frac{s_{1}}{2}+\frac{d}{4}} \left\|u^{\ell}(\tau)\right\|_{\dot{B}^{\frac{d}{p}}_{p,1}}
\lesssim \mathcal{D}_{p}(t)
\end{equation}
and also that, owing to $-s_{1}-1<-s_{1}<1-s_{1}<\frac{d}{2}$ and $-s_{1}<2-s_{1}<\frac{d}{2}+1$,
\begin{equation}\label{Eq:3.12}
\left\|(a^{\ell},\nabla u^{\ell})(\tau)\right\|_{\dot{B}^{-s_{1}}_{2,1}}
\lesssim \langle\tau\rangle^{-\frac{1}{2}} \mathcal{D}_{p}(\tau), \ \ \
\left\|(\nabla a^{\ell},\nabla^{2}u^{\ell})(\tau)\right\|_{\dot{B}^{-s_{1}}_{2,1}}
\lesssim \langle\tau\rangle^{-1} \mathcal{D}_{p}(\tau).
\end{equation}
Observe that $-s_{1}-1<1-s_{1}<\frac{d}{2}$, $\frac{d}{p}-\frac{d}{2}-s_{1}<\frac{d}{p}-1$ and $\alpha>1$ for small $\varepsilon>0$, we have by embedding,
\begin{eqnarray}\label{Eq:3.13}
\left\|(\nabla a,a^{h},u^{h})(\tau)\right\|_{\dot{B}^{\frac{d}{p}-\frac{d}{2}-s_{1}}_{p,1}}
&\lesssim& \left\|\nabla a^{\ell}(\tau)\right\|_{\dot{B}^{-s_{1}}_{2,1}}
+\left\|(\nabla a^{h}, a^{h},u^{h})(\tau)\right\|_{\dot{B}^{\frac{d}{p}-1}_{p,1}}\nonumber\\
&\lesssim& \langle\tau\rangle^{-1} \mathcal{D}_{p}(\tau)
\end{eqnarray}
and thus, thanks to $-s_{1}<\frac{d}{2}-1\leq\frac{d}{p}<\frac{d}{p}+1\leq\frac{d}{2}+1$ for all $p\leq\frac{2d}{d-2}$,
\begin{equation}\label{Eq:3.14}
\sup_{\tau \in [0,t]}\langle\tau\rangle^{\frac{s_{1}}{2}+\frac{d}{2p}+\frac{1}{2}}\left\|\nabla u^{\ell} (\tau)\right\|_{\dot{B}^{\frac{d}{p}}_{2,1}}
+\sup_{\tau \in [0,t]}\langle\tau\rangle^{\frac{s_{1}}{2}+\frac{d}{2p}}\left\|u^{\ell} (\tau)\right\|_{\dot{B}^{\frac{d}{p}}_{2,1}}
\lesssim \mathcal{D}_{p}(t).
\end{equation}
Now, let us begin with prove \eqref{Eq:3.8}.
We decompose $\Lambda^{-1}\mathrm{div}\,(au^{\ell})=\Lambda^{-1}\mathrm{div}\,(a^{\ell}u^{\ell})+\Lambda^{-1}\mathrm{div}\,(a^{h}u^{\ell})$
and $u\cdot\nabla u^{\ell}=u^{\ell}\cdot\nabla u^{\ell}+u^{h}\cdot\nabla u^{\ell}$.
To bound the term with $\Lambda^{-1}\mathrm{div}\,(a^{\ell}u^{\ell})$, we note that, due to \eqref{Eq:3.9}, \eqref{Eq:3.11} and \eqref{Eq:3.12},
\begin{eqnarray*}
\int_{0}^{t}\langle t-\tau\rangle^{-\frac{s_{1}+s}{2}}\left\|\Lambda^{-1}\mathrm{div}\,(a^{\ell}u^{\ell})\right\|^{\ell}_{\dot{B}^{-s_{1}}_{2,\infty}}d\tau
&\lesssim& \int_{0}^{t}\langle t-\tau\rangle^{-\frac{s_{1}+s}{2}}\left\|a^{\ell}\right\|_{\dot{B}^{-s_{1}}_{2,1}}
\left\|u^{\ell}\right\|_{\dot{B}^{\frac{d}{p}}_{p,1}}d\tau\\
&\lesssim &\mathcal{D}^{2}_{p}(t)\int_{0}^{t}\langle t-\tau\rangle^{-\frac{s_{1}+s}{2}}
\langle \tau\rangle^{-\frac{s_{1}}{2}-\frac{d}{4}-\frac{1}{2}} d\tau.
\end{eqnarray*}
Owing to $\frac{s_{1}}{2}+\frac{d}{4}+\frac{1}{2}>1$ and $\frac{s_{1}}{2}+\frac{d}{4}+\frac{1}{2}\geq\frac{s_{1}+s}{2}$ for $s_{1}$ fulfilling \eqref{Eq:1.8} and $s\leq\frac{d}{2}+1$, inequality \eqref{Eq:3.1} ensures that
\begin{equation*}
\int_{0}^{t}\langle t-\tau\rangle^{-\frac{s_{1}+s}{2}}\left\|\Lambda^{-1}\mathrm{div}\,(a^{\ell}u^{\ell})\right\|^{\ell}_{\dot{B}^{-s_{1}}_{2,\infty}}d\tau
\lesssim \langle t\rangle^{-\frac{s_{1}+s}{2}}\mathcal{D}^{2}_{p}(t).
\end{equation*}
Bounding $u^{\ell}\cdot \nabla u^{\ell}$ essentially follows from the same procedure as $\Lambda^{-1}\mathrm{div}\,(a^{\ell}u^{\ell})$, we omit it.
To handle the term containing $\Lambda^{-1}\mathrm{div}\left(a^{h}u^{\ell}\right)$, we write that, thanks to \eqref{Eq:3.10},
\begin{equation*}
\int_{0}^{t}\langle t-\tau\rangle^{-\frac{s_{1}+s}{2}}\left\|\Lambda^{-1}\mathrm{div}(a^{h}u^{\ell})\right\|^{\ell}_{\dot{B}^{\frac{d}{p}-\frac{d}{2}-s_{1}}_{2,\infty}}d\tau
\lesssim \int_{0}^{t}\langle t-\tau\rangle^{-\frac{s_{1}+s}{2}}\left\|a^{h}\right\|_{\dot{B}^{\frac{d}{p}-\frac{d}{2}-s_{1}}_{p,1}}
\left\|u^{\ell}\right\|_{\dot{B}^{\frac{d}{p}}_{2,1}}d\tau.
\end{equation*}
Observe that, as $s_{1}\leq s_{1}+\frac{d}{2}-\frac{d}{p}$ (for $p\geq2$), we get
\begin{equation*}
\left\|\Lambda^{-1}\mathrm{div}\,(a^{h}u^{\ell})\right\|^{\ell}_{\dot {B}^{-s_{1}}_{2,\infty}}
\lesssim \left\|\Lambda^{-1}\mathrm{div}\,(a^{h}u^{\ell})\right\|^{\ell}_{\dot {B}^{\frac{d}{p}-\frac{d}{2}-s_{1}}_{2,\infty}}.
\end{equation*}
Using the relations $\frac{s_{1}}{2}+\frac{d}{2p}+1 \geq \frac{s_{1}}{2}+\frac{d}{4}+\frac{1}{2}>1
$ and $\frac{s_{1}}{2}+\frac{d}{2p}+1\geq\frac{s_{1}+s}{2}$ for all $s\leq\frac{d}{2}+1$ and $p\leq\frac{2d}{d-2}$
as well as \eqref{Eq:3.13}, \eqref{Eq:3.14} and \eqref{Eq:3.1}, we deduce that
\begin{eqnarray*}
\int_{0}^{t}\langle t-\tau\rangle^{-\frac{s_{1}+s}{2}}\left\|\Lambda^{-1}\mathrm{div}\,(a^{h}u^{\ell})\right\|^{\ell}_{\dot{B}^{-s_{1}}_{2,\infty}}d\tau
&\lesssim &\mathcal{D}^{2}_{p}(t)\int_{0}^{t}\langle t-\tau\rangle^{-\frac{s_{1}+s}{2}}
\langle \tau\rangle^{-\frac{s_{1}}{2}-\frac{d}{2p}-1} d\tau\\
&\lesssim& \langle t\rangle^{-\frac{s_{1}+s}{2}}\mathcal{D}^{2}_{p}(t).
\end{eqnarray*}
In light of the relation $s_{1}\leq s_{1}+\frac{d}{2}-\frac{d}{p}$ (for $p\geq2$) and using \eqref{Eq:3.10}, \eqref{Eq:3.13}, \eqref{Eq:3.14} and \eqref{Eq:3.1}, we end up with
\begin{eqnarray*}
\int_{0}^{t}\langle t-\tau\rangle^{-\frac{s_{1}+s}{2}}\left\|u^{h}\cdot\nabla u^{\ell}\right\|^{\ell}_{\dot {B}^{-s_{1}}_{2,\infty}}d\tau
&\lesssim& \int_{0}^{t}\langle t-\tau\rangle^{-\frac{s_{1}+s}{2}}\left\|u^{h}\right\|_{\dot {B}^{\frac{d}{p}-\frac{d}{2}-s_{1}}_{p,1}}
\left\|\nabla u^{\ell}\right\|_{\dot {B}^{\frac {d}{p}}_{2,1}}d\tau\\
&\lesssim& \mathcal{D}^{2}_{p}(t)\int_{0}^{t}\langle t-\tau\rangle^{-\frac{s_{1}+s}{2}}
\langle\tau\rangle^{-\frac{s_{1}}{2}-\frac{d}{2p}-\frac{3}{2}}d\tau\\
& \lesssim &\langle t\rangle^{-\frac{s_{1}+s}{2}} \mathcal{D}^{2}_{p}(t).
\end{eqnarray*}
For the term with $k(a)\nabla a^{\ell}$, we take advantage of Proposition \ref{Prop2.4}, \eqref{Eq:3.9}, \eqref{Eq:3.11}, \eqref{Eq:3.12} and \eqref{Eq:3.1}, and similarly get
\begin{eqnarray*}
\int_{0}^{t}\langle t-\tau\rangle^{-\frac{s_{1}+s}{2}}\left\|k(a)\nabla a^{\ell}\right\|^{\ell}_{\dot{B}^{-s_{1}}_{2,\infty}}d\tau
&\lesssim & \int_{0}^{t}\langle t-\tau\rangle^{-\frac{s_{1}+s}{2}}\left\|a\right\|_{\dot{B}^{\frac{d}{p}}_{p,1}}
\left\|\nabla a^{\ell}\right\|_{\dot{B}^{-s_{1}}_{2,1}}d\tau\\
&\lesssim& \mathcal{D}^{2}_{p}(t)\int_{0}^{t}\langle t-\tau\rangle^{-\frac{s_{1}+s}{2}}
\langle \tau\rangle^{-\frac{s_{1}}{2}-\frac{d}{4}-\frac{3}{2}}d\tau\\
&\lesssim& \langle t\rangle^{-\frac{s_{1}+s}{2}} \mathcal{D}^{2}_{p}(t).
\end{eqnarray*}
Let us next look at the term with $g_{3}(a,u^{\ell})$, it follows from Proposition \ref{Prop2.4}, \eqref{Eq:3.9}, \eqref{Eq:3.11}, \eqref{Eq:3.12} and \eqref{Eq:3.1} that
\begin{eqnarray*}
\int_{0}^{t}\langle t-\tau\rangle^{-\frac{s_{1}+s}{2}}\left\|g_{3}(a,u^{\ell})\right\|^{\ell}_{\dot{B}^{-s_{1}}_{2,\infty}}d\tau
&\lesssim & \int_{0}^{t}\langle t-\tau\rangle^{-\frac{s_{1}+s}{2}}\left\|a\right\|_{\dot{B}^{\frac{d}{p}}_{p,1}}
\left\|\nabla^{2} u^{\ell}\right\|_{\dot{B}^{-s_{1}}_{2,1}}d\tau\\
&\lesssim& \mathcal{D}^{2}_{p}(t)\int_{0}^{t}\langle t-\tau\rangle^{-\frac{s_{1}+s}{2}}\langle \tau\rangle^{-\frac{s_{1}}{2}-\frac{d}{4}-\frac{3}{2}}d\tau\\
&\lesssim& \langle t\rangle^{-\frac{s_{1}+s}{2}} \mathcal{D}^{2}_{p}(t).
\end{eqnarray*}
Let us finally bound the term involving $g_{4}(a,u^{\ell})$. Keeping in mind that $s_{1}\leq s_{1}+\frac{d}{2}-\frac{d}{p}$ (for $p\geq2$) and using \eqref{Eq:3.10}, Proposition \ref{Prop2.4} and \eqref{Eq:3.2}, we arrive at
\begin{equation*}
\int_{0}^{t}\langle t-\tau\rangle^{-\frac{s_{1}+s}{2}}\left\|g_{4}(a,u^{\ell})\right\|^{\ell}_{\dot{B}^{-s_{1}}_{2,\infty}}d\tau
\lesssim \int_{0}^{t}\langle t-\tau\rangle^{-\frac{s_{1}+s}{2}}
\left\|\nabla a\right\|_{\dot{B}^{\frac{d}{p}-\frac{d}{2}-s_{1}}_{p,1}}
\left\|\nabla u^{\ell}\right\|_{\dot{B}^{\frac{d}{p}}_{2,1}}d\tau,
\end{equation*}
where the relation $\frac{d}{p}-\frac{d}{2}-s_{1}+1>1-\frac{d}{p}>-\frac{d}{p}$ implies that $\frac{d}{p}-\frac{d}{2}-s_{1}+1$
satisfies the regularity requirement of Proposition \ref{Prop2.4}. With the aid of \eqref{Eq:3.13}-\eqref{Eq:3.14} and \eqref{Eq:3.1}, we conclude that
\begin{eqnarray*}
\int_{0}^{t}\langle t-\tau\rangle^{-\frac{s_{1}+s}{2}}\left\|g_{4}(a,u^{\ell})\right\|^{\ell}_{\dot{B}^{-s_{1}}_{2,\infty}}d\tau
&\lesssim& \mathcal{D}^{2}_{p}(t)\int_{0}^{t}\langle t-\tau\rangle^{-\frac{s_{1}+s}{2}}
\langle\tau\rangle^{-\frac{s_{1}}{2}-\frac{d}{2p}-\frac{3}{2}}d\tau\\
&\lesssim &\langle t\rangle^{-\frac{s_{1}+s}{2}}\mathcal{D}^{2}_{p}(t).
\end{eqnarray*}
Therefore, putting all estimates together leads to \eqref{Eq:3.8}.
\end{proof}
To bound nonlinear terms in $\Lambda^{-1}f^{h}$ and $g^{h}$, precisely,
$$\Lambda^{-1}\mathrm{div}\,(au^{h}),\ \ u\cdot\nabla u^{h}, \ \ k(a)\nabla a^{h}, \ \ g_{3}(a,u^{h}) \ \ \hbox{and} \ \ g_{4}(a,u^{h}),$$
we shall proceed these calculations differently depending on whether $2\leq p\leq d$ or $p>d$. We claim that
\begin{equation}\label{Eq:3.15}
\left\|a(\tau)\right\|_{\dot{B}^{\frac{d}{p}-1}_{p,1}} \lesssim \langle\tau\rangle^{-\frac{s_{1}}{2}-\frac{d}{4}}\mathcal{D}_{p}(\tau), \ \ \ \
\left\|u(\tau)\right\|_{\dot{B}^{\frac{d}{p}-1}_{p,1}} \lesssim \langle\tau\rangle^{-\frac{s_{1}}{2}-\frac{d}{4}+\frac{1}{2}}\mathcal{D}_{p}(\tau)
\end{equation}
and also that
\begin{equation}\label{Eq:3.16}
\left\|(\nabla a^{h},u^{h})(\tau)\right\|_{\dot{B}^{\frac{d}{p}-1}_{p,1}}
\lesssim \langle\tau\rangle^{-\alpha}\mathcal{D}_{p}(\tau), \ \ \left\|( u^{h},\nabla u^{h})(\tau)\right\|_{\dot{B}^{\frac{d}{p}}_{p,1}}
\lesssim \tau^{-\alpha}\mathcal{D}_{p}(\tau)
\end{equation}
for all $\tau\geq0$.

Indeed, as for \eqref{Eq:3.15}, owing to the relations $-s_{1}-1<-s_{1}<\frac{d}{2}-1<\frac{d}{2}$ and $\alpha \geq \frac{s_{1}}{2}+\frac{d}{4}$ for small enough $\varepsilon>0$, the embedding and the definition of $\mathcal{D}_{p}(t)$, we obtain
\begin{eqnarray*}
\left\|a(\tau)\right\|_{\dot{B}^{\frac{d}{p}-1}_{p,1}} &\lesssim& \langle\tau\rangle^{-\frac{s_{1}}{2}-\frac{d}{4}}\left(\langle\tau\rangle^{\frac{s_{1}}{2}+\frac{d}{4}}\left\|a(\tau)\right\|^{\ell}_{\dot{B}^{\frac{d}{2}-1}_{2,1}}
+\langle\tau\rangle^{\alpha}\left\|a(\tau)\right\|^{h}_{\dot{B}^{\frac{d}{p}}_{p,1}}\right)\\
&\lesssim& \langle\tau\rangle^{-\frac{s_{1}}{2}-\frac{d}{4}}\mathcal{D}_{p}(\tau)
\end{eqnarray*}
and
\begin{eqnarray*}
\left\|u(\tau)\right\|_{\dot{B}^{\frac{d}{p}-1}_{p,1}} &\lesssim& \langle\tau\rangle^{-\frac{s_{1}}{2}-\frac{d}{4}+\frac{1}{2}}\left(\langle\tau\rangle^{\frac{s_{1}}{2}+\frac{d}{4}-\frac{1}{2}}\left\|u(\tau)\right\|^{\ell}_{\dot{B}^{\frac{d}{2}-1}_{2,1}}
+\langle\tau\rangle^{\alpha}\left\|u(\tau)\right\|^{h}_{\dot{B}^{\frac{d}{p}-1}_{p,1}}\right)\\
&\lesssim& \langle\tau\rangle^{-\frac{s_{1}}{2}-\frac{d}{4}+\frac{1}{2}}\mathcal{D}_{p}(\tau).
\end{eqnarray*}
The inequality \eqref{Eq:3.16} is easily followed by the definition of $\mathcal{D}_{p}(t)$.
\begin{lem}\label{Lem3.2} If $p$ satisfies \eqref{Eq:1.5} and $2\leq p\leq d$, then it holds that for all $t\geq0$,
\begin{equation}\label{Eq:3.17}
\int^{t}_{0}\langle t-\tau\rangle^{-\frac{s_{1}+s}{2}}\left\|\left(\Lambda^{-1}f^{h},g^{h}\right)\right\|^{\ell}_{\dot{B}^{-s_{1}}_{2,\infty}}d\tau
\lesssim \langle t\rangle^{-\frac{s_{1}+s}{2}}\left(\mathcal{D}^{2}_{p}(t)+\mathcal{E}^{2}_{p}(t)\right).
\end{equation}
\end{lem}
\begin{proof}
In order to prove \eqref{Eq:3.17}, we develop the following inequality
\begin{equation}\label{Eq:3.18}
\|FG^{h}\|^{\ell}_{\dot{B}^{-s_{1}}_{2,\infty}}\lesssim \|FG^{h}\|^{\ell}_{\dot{B}^{-s_{0}}_{2,\infty}}
\lesssim \|F\|_{\dot{B}^{\frac{d}{p}-1}_{p,1}}
\|G^{h}\|_{\dot{B}^{\frac{d}{p}-1}_{p,1}}
\end{equation}
for $1-\frac{d}{2}<s_{1}\leq s_{0}$ and $2\leq p\leq d$.

Indeed, if $2\leq p<d$, then $\frac{d}{p}-1>0$. It follows from the last item of Proposition \ref{Prop2.2} with $p_{1}=p_{2}=p$, $\sigma=\frac{d}{p}-1$ and $\frac 1q=\frac 1p+\frac 1d$, that
$$\|FG^{h}\|_{\dot{B}^{1-\frac{d}{p}}_{q,\infty}}\lesssim \|F\|_{\dot{B}^{\frac{d}{p}-1}_{p,1}}\|G^{h}\|_{\dot{B}^{1-\frac{d}{p}}_{p,\infty}}.$$
The fact $p\leq \frac {2d}{d-2}$ implies that $q\leq2$. Hence, \eqref{Eq:3.18} directly stems from
$\dot{B}^{1-\frac{d}{p}}_{q,\infty}\hookrightarrow\dot{B}^{-s_{0}}_{2,\infty}$ and $1-\frac{d}{p}<\frac{d}{p}-1$. If $p=d$, owing to the embeddings $L^{\frac{p}{2}}\hookrightarrow \dot{B}^{-s_{0}}_{2,\infty}$, $\dot{B}^{0}_{p,1}\hookrightarrow L^{p}$ and H\"{o}lder inequality, one has
\begin{equation*}
\|FG^{h}\|^{\ell}_{\dot{B}^{-s_{1}}_{2,\infty}}
\lesssim \|FG^{h}\|^{\ell}_{\dot{B}^{-s_{0}}_{2,\infty}}
\lesssim \|FG^{h}\|_{L^{\frac{p}{2}}}\lesssim
\|F\|_{L^{p}} \|G^{h}\|_{L^{p}}
\lesssim\|F\|_{\dot{B}^{0}_{p,1}} \|G^{h}\|_{\dot{B}^{0}_{p,1}}.
\end{equation*}
Now, we begin to estimates those terms in $\Lambda^{-1}f^{h}$ and $g^{h}$. For the term with $\Lambda^{-1}\mathrm{div}\,(a\,u^{h})$, we use \eqref{Eq:3.18}, \eqref{Eq:3.15} and \eqref{Eq:3.16}, and get
\begin{eqnarray*}
\int_{0}^{t}\langle t-\tau\rangle^{-\frac{s_{1}+s}{2}}\left\|\Lambda^{-1}\mathrm{div}\,(au^{h})\right\|^{\ell}_{\dot{B}^{-s_{1}}_{2,\infty}}d\tau
&\lesssim& \int_{0}^{t}\langle t-\tau\rangle^{-\frac{s_{1}+s}{2}}\left\|a\right\|_{\dot{B}^{\frac{d}{p}-1}_{p,1}}
\left\|u^{h}\right\|_{\dot{B}^{\frac{d}{p}-1}_{p,1}}d\tau \\
&\lesssim &\mathcal{D}^{2}_{p}(t)\int_{0}^{t}\langle t-\tau\rangle^{-\frac{s_{1}+s}{2}}\langle\tau\rangle^{-\frac{s_{1}}{2}-\frac{d}{4}-\alpha} d\tau.
\end{eqnarray*}
Thanks to $\frac{s_{1}}{2}+\frac{d}{4}+\alpha>1$ and $\frac{s_{1}}{2}+\frac{d}{4}+\alpha\geq\frac{s_{1}+s}{2}$
for small enough $\varepsilon>0$ and $s\leq\frac{d}{2}+1$, inequality \eqref{Eq:3.1} implies that
\begin{equation*}
\int_{0}^{t}\langle t-\tau\rangle^{-\frac{s_{1}+s}{2}}\left\|\Lambda^{-1}\mathrm{div}\,(au^{h})\right\|^{\ell}_{\dot{B}^{-s_{1}}_{2,\infty}}d\tau
\lesssim \langle t\rangle^{-\frac{s_{1}+s}{2}}\mathcal{D}^{2}_{p}(t).
\end{equation*}
For the term with $k(a)\nabla a^{h}$, we use \eqref{Eq:3.18}, \eqref{Eq:3.2} and Proposition \ref{Prop2.4} with $\frac{d}{p}-1>-\frac{d}{p}$, and get
\begin{equation}\label{Eq:3.19}
\left\|k(a)\nabla a^{h}\right\|^{\ell}_{\dot{B}^{-s_{1}}_{2,\infty}}
\lesssim\left\|k(a)\right\|_{\dot{B}^{\frac{d}{p}-1}_{p,1}}\left\|\nabla a^{h}\right\|_{\dot{B}^{\frac{d}{p}-1}_{p,1}}
\lesssim\left\|a\right\|_{\dot{B}^{\frac{d}{p}-1}_{p,1}}\left\|\nabla a^{h}\right\|_{\dot{B}^{\frac{d}{p}-1}_{p,1}}.
\end{equation}
Then we use \eqref{Eq:3.15}, \eqref{Eq:3.16} and \eqref{Eq:3.1} to conclude that we end up with
\begin{eqnarray*}
\int_{0}^{t}\langle t-\tau\rangle^{-\frac{s_{1}+s}{2}}\left\|k(a)\nabla a^{h}\right\|^{\ell}_{\dot{B}^{-s_{1}}_{2,\infty}}d\tau
&\lesssim&\int_{0}^{t}\langle t-\tau\rangle^{-\frac{s_{1}+s}{2}}
\left\|a\right\|_{\dot{B}^{\frac{d}{p}-1}_{p,1}}
\left\|\nabla a^{h}\right\|_{\dot{B}^{\frac{d}{p}-1}_{p,1}}d\tau\\
&\lesssim& \mathcal{D}^{2}_{p}(t)\int_{0}^{t}\langle t-\tau\rangle^{-\frac{s_{1}+s}{2}}
\langle\tau\rangle^{-\frac{s_{1}}{2}-\frac{d}{4}-\alpha} d\tau\\
&\lesssim &\langle t\rangle^{-\frac{s_{1}+s}{2}} \mathcal{D}^{2}_{p}(t).
\end{eqnarray*}
To bound the term involving $u\cdot\nabla u^{h}$, we write that, thanks to \eqref{Eq:3.18},
\begin{eqnarray*}
\int_{0}^{t}\langle t-\tau\rangle^{-\frac{s_{1}+s}{2}}
\left\|u\cdot\nabla u^{h}\right\|^{\ell}_{\dot{B}^{-s_{1}}_{2,\infty}}d\tau
&\lesssim& \int_{0}^{t}\langle t-\tau\rangle^{-\frac{s_{1}+s}{2}}
\left\|u\right\|_{\dot{B}^{\frac{d}{p}-1}_{p,1}}
\left\| u^{h}\right\|_{\dot{B}^{\frac{d}{p}}_{p,1}} d\tau\\
&=& \left(\int_{0}^{1}+\int_{1}^{t}\right)\left(\cdots\right)d\tau\triangleq I_{1}+I_{2}.
\end{eqnarray*}
It is clear that $I_{1}\lesssim \langle t\rangle^{-\frac{s_{1}+s}{2}} \mathcal{E}^{2}_{p}(1)$ and that, owing to \eqref{Eq:3.15}, \eqref{Eq:3.16} and \eqref{Eq:3.1}, if $t\geq1$,
\begin{eqnarray*}
I_{2}&\lesssim & \int_{1}^{t}\langle t-\tau\rangle^{-\frac{s_{1}+s}{2}}
\langle\tau\rangle^{-\frac{s_{1}}{2}-\frac{d}{4}+\frac{1}{2}-\alpha}
\left(\langle\tau\rangle^{\frac{s_{1}}{2}+\frac{d}{4}-\frac{1}{2}}\left\|u\right\|_{\dot{B}^{\frac{d}{p}-1}_{p,1}}\right)
\left(\tau ^{\alpha}\left\| u^{h}\right\|_{\dot{B}^{\frac{d}{p}}_{p,1}}\right)d\tau\\
&\lesssim& \left(\sup_{\tau \in [1,t]}\langle\tau\rangle^{\frac{s_{1}}{2}+\frac{d}{4}-\frac{1}{2}}\left\|u\right\|_{\dot{B}^{\frac{d}{p}-1}_{p,1}}\right)
\left(\sup_{\tau \in [1,t]}\tau ^{\alpha}\left\|u^{h}\right\|_{\dot{B}^{\frac{d}{p}}_{p,1}}\right)\\
&&\times\int_{1}^{t}\langle t-\tau\rangle^{-\frac{s_{1}+s}{2}}
\langle\tau\rangle^{-\frac{s_{1}}{2}-\frac{d}{4}+\frac{1}{2}-\alpha}d\tau
\lesssim \langle t\rangle^{-\frac{s_{1}+s}{2}}\mathcal{D}^{2}_{p}(t),
\end{eqnarray*}
where we noticed that $\frac{s_{1}}{2}+\frac{d}{4}-\frac{1}{2}+\alpha>1$ and $\frac{s_{1}}{2}+\frac{d}{4}-\frac{1}{2}+\alpha \geq\frac{s_{1}+s}{2}$ for all $s\leq\frac{d}{2}+1$.

For the term with $g_{3}(a,u^{h})$, following \eqref{Eq:3.19}, we arrive at
\begin{eqnarray*}
\int_{0}^{t}\langle t-\tau\rangle^{-\frac{s_{1}+s}{2}}\left\|g_{3}(a,u^{h})\right\|^{\ell}_{\dot{B}^{-s_{1}}_{2,\infty}}d\tau
&\lesssim & \int_{0}^{t}\langle t-\tau\rangle^{-\frac{s_{1}+s}{2}}\left\|a\right\|_{\dot{B}^{\frac{d}{p}-1}_{p,1}}
\left\|\nabla u^{h}\right\|_{\dot{B}^{\frac{d}{p}}_{p,1}}d\tau\\
&=& \left(\int_{0}^{1}+\int_{1}^{t}\right)\left(\cdots\right)d\tau\triangleq J_{1}+J_{2}.
\end{eqnarray*}
It follows from the definition of $\mathcal{E}(t)$ that
$$J_{1}\lesssim \langle t\rangle^{-\frac{s_{1}+s}{2}}\mathcal{E}^{2}_{p}(1).$$
By using \eqref{Eq:3.15}, \eqref{Eq:3.16} and \eqref{Eq:3.1}, we get, if $t\geq1$,
\begin{eqnarray*}
J_{2}&\lesssim & \int_{1}^{t}\langle t-\tau\rangle^{-\frac{s_{1}+s}{2}}
\langle\tau\rangle^{-\frac{s_{1}}{2}-\frac{d}{4}-\alpha}
\left(\langle\tau\rangle^{\frac{s_{1}}{2}+\frac{d}{4}}\left\|a\right\|_{\dot{B}^{\frac{d}{p}-1}_{p,1}}\right)
\left(\tau^{\alpha}\left\| \nabla u^{h}\right\|_{\dot{B}^{\frac{d}{p}}_{p,1}}\right)d\tau\\
&\lesssim& \left(\sup_{\tau\in[1,t]} \langle\tau\rangle^{\frac{s_{1}}{2}+\frac{d}{4}}\left\|a\right\|_{\dot{B}^{\frac{d}{p}-1}_{p,1}}\right)
\left(\sup_{\tau\in[1,t]}\tau^{\alpha}\left\|\nabla u^{h}\right\|_{\dot{B}^{\frac{d}{p}}_{p,1}}\right)\\
&&\times \int_{1}^{t}\langle t-\tau\rangle^{-\frac{s_{1}+s}{2}}
\langle\tau\rangle^{-\frac{s_{1}}{2}-\frac{d}{4}-\alpha}d\tau
\lesssim\langle t\rangle^{-\frac{s_{1}+s}{2}}\mathcal{D}^{2}_{p}(t).
\end{eqnarray*}
Let us finally look at the term with $g_{4}(a,u^{h})$. We observe that, owing to \eqref{Eq:3.18} and Proposition \ref{Prop2.4},
\begin{eqnarray*}
\int_{0}^{t}\langle t-\tau\rangle^{-\frac{s_{1}+s}{2}}\left\|g_{4}(a,u^{h})\right\|^{\ell}_{\dot{B}^{-s_{1}}_{2,\infty}}d\tau
&\lesssim& \int_{0}^{t}\langle t-\tau\rangle^{-\frac{s_{1}+s}{2}}
\left\|a\right\|_{\dot{B}^{\frac{d}{p}}_{p,1}}\left\| u^{h}\right\|_{\dot{B}^{\frac{d}{p}}_{p,1}}d\tau\\
&=& \left(\int_{0}^{1}+\int_{1}^{t}\right)\left(\cdots\right)d\tau\triangleq K_{1}+K_{2}.
\end{eqnarray*}
It is clear that $K_{1}\lesssim \langle t\rangle^{-\frac{s_{1}+s}{2}} \mathcal{E}^{2}_{p}(1)$ and that, owing to \eqref{Eq:3.11}, \eqref{Eq:3.16} and
\eqref{Eq:3.1}, if $t\geq1$,
\begin{eqnarray*}
K_{2}&= & \int_{1}^{t}\langle t-\tau\rangle^{-\frac{s_{1}+s}{2}}
\langle\tau\rangle^{-\frac{s_{1}}{2}-\frac{d}{2}-\frac{1}{2}-\alpha}
\left(\langle \tau\rangle^{\frac{s_{1}}{2}+\frac{d}{4}+\frac{1}{2}}\left\|a\right\|_{\dot{B}^{\frac{d}{p}}_{p,1}}\right)
\left(\tau ^{\alpha}\left\|u^{h}\right\|_{\dot{B}^{\frac{d}{p}}_{p,1}}\right)d\tau\\
&\lesssim& \left(\sup_{\tau \in[1,t]}\langle \tau\rangle^{\frac{s_{1}}{2}+\frac{d}{4}+\frac{1}{2}}\left\|a\right\|_{\dot{B}^{\frac{d}{p}}_{p,1}}\right)
\left(\sup_{\tau \in [1,t]} \tau ^{\alpha}\left\|u^{h}\right\|_{\dot{B}^{\frac{d}{p}}_{p,1}}\right)\\
&&\times \int_{1}^{t}\langle t-\tau\rangle^{-\frac{s_{1}+s}{2}}
\langle\tau\rangle^{-\frac{s_{1}}{2}-\frac{d}{4}-\frac{1}{2}-\alpha}d\tau
\lesssim  \langle t\rangle^{-\frac{s_{1}+s}{2}}\mathcal{D}^{2}_{p}(t).
\end{eqnarray*}
Hence, the proof of Lemma \ref{Lem3.2} is complete.
\end{proof}
\begin{lem}\label{Lem3.3} Let $p$ satisfy \eqref{Eq:1.5} and $p>d$. It holds that for all $t\geq0$,
\begin{equation}\label{Eq:3.20}
\int^{t}_{0}\langle t-\tau\rangle^{-\frac{s_{1}+s}{2}}\left\|\left(\Lambda^{-1}f^{h},g^{h}\right)\right\|^{\ell}_{\dot{B}^{-s_{1}}_{2,\infty}}d\tau
\lesssim \langle t\rangle^{-\frac{s_{1}+s}{2}}\left(\mathcal{D}^{2}_{p}(t)+\mathcal{E}^{2}_{p}(t)\right).
\end{equation}
\end{lem}
\begin{proof}
Let us end the first step in the case where \eqref{Eq:1.8} is fulfilled and $p>d$. Using the fact that $s_{1}\leq s_{0}$ and
applying \eqref{Eq:2.4} with $\sigma=1-\frac{d}{p}$ yield
\begin{equation}\label{Eq:3.21}
\left\|F G^{h}\right\|^{\ell}_{\dot{B}^{-s_{1}}_{2,\infty}}\lesssim
\left\|F G^{h}\right\|^{\ell}_{\dot{B}^{-s_{0}}_{2,\infty}}\lesssim \left(\left\|F\right\|_{\dot{B}^{1-\frac{d}{p}}_{p,1}}+\left\|\dot{S}_{j_{0}+N_{0}}F\right\|_{L^{p^{*}}}\right)\left\|G^{h}\right\|_{\dot{B}^{\frac{d}{p}-1}_{p,1}}
\end{equation}
with $\frac{1}{p^{*}}\triangleq\frac{1}{2}-\frac{1}{p}$ and thus, because $\dot{B}^{\frac{d}{p}}_{2,1}\hookrightarrow L^{p^{*}}$,
\begin{equation}\label{Eq:3.22}
\left\|F G^{h}\right\|^{\ell}_{\dot{B}^{-s_{1}}_{2,\infty}}\lesssim
\left\|F G^{h}\right\|^{\ell}_{\dot{B}^{-s_{0}}_{2,\infty}}\lesssim \left(\left\|F^{\ell}\right\|_{\dot{B}^{\frac{d}{p}}_{2,1}}+\left\|F\right\|_{\dot{B}^{1-\frac{d}{p}}_{p,1}}
\right)\left\|G^{h}\right\|_{\dot{B}^{\frac{d}{p}-1}_{p,1}}.
\end{equation}
Moreover, in the case $p>d$, we have
\begin{equation} \label{Eq:3.23}
\left\|a(\tau)\right\|_{\dot{B}^{1-\frac{d}{p}}_{p,1}}
\lesssim \langle\tau\rangle^{-\frac{s_{1}+2-s_{0}}{2}}\mathcal{D}_{p}(\tau),\
\left\|a^{\ell}(\tau)\right\|_{\dot{B}^{\frac{d}{p}}_{2,1}}
\lesssim \langle\tau\rangle^{-\frac{s_{1}}{2}-\frac{d}{2p}-\frac{1}{2}}\mathcal{D}_{p}(\tau)
 \ \ \hbox{for} \ \ \tau\geq0.
\end{equation}
Indeed, due to the embedding and the definition of $\mathcal{D}_{p}(t)$, we obtain
\begin{eqnarray*}
\left\|a^{\ell}(\tau)\right\|_{\dot{B}^{1-\frac{d}{p}}_{p,1}}
&\lesssim& \left\|a^{\ell}(\tau)\right\|_{\dot{B}^{1-s_{0}}_{2,1}}
\lesssim \langle\tau\rangle^{-\frac{s_{1}+2-s_{0}}{2}} \mathcal{D}_{p}(\tau),
\end{eqnarray*}
where the relation $p>d$ implies that $-s_{1}-1<-s_{1}<\frac{d}{2}-1<1-s_{0}<\frac{d}{2}$.

As $1-\frac{d}{p}<\frac{d}{p}$, it holds that
\begin{equation*}
\left\|a^{h}(\tau)\right\|_{\dot{B}^{1-\frac{d}{p}}_{p,1}}
\lesssim \left\|a^{h}(\tau)\right\|_{\dot{B}^{\frac{d}{p}}_{p,1}}
\lesssim \langle\tau\rangle^{-\alpha} \mathcal{D}_{p}(\tau),
\end{equation*}
and the first inequality of \eqref{Eq:3.23} is thus satisfied by $a^{h}$ owing to $\frac{s_{1}+2-s_{0}}{2}<\frac{s_{1}}{2}+\frac{d}{4}+\frac{1}{2}\leq\alpha$ for small enough $\varepsilon>0$. Thanks to $-s_{1}-1<-s_{1}<\frac{d}{2}-1\leq\frac{d}{p}<\frac{d}{2}$ for all $p\leq \frac{2d}{d-2}$,
the second inequality of \eqref{Eq:3.23} is easily followed by the definition of $\mathcal{D}_{p}(t)$.

Taking $F=a$ and $G=u$ in \eqref{Eq:3.22} and using \eqref{Eq:3.23} and \eqref{Eq:3.16}, we get
\begin{eqnarray*}
&&\int_{0}^{t}\langle t-\tau\rangle^{-\frac{s_{1}+s}{2}}\left\|\Lambda^{-1}\mathrm{div}\,(au^{h})\right\|^{\ell}_{\dot{B}^{-s_{1}}_{2,\infty}}d\tau\\
&\lesssim & \int_{0}^{t}\langle t-\tau\rangle^{-\frac{s_{1}+s}{2}}\left(\left\|a^{\ell}\right\|_{\dot{B}^{\frac{d}{p}}_{2,1}}+\left\|a\right\|_{\dot{B}^{1-\frac{d}{p}}_{p,1}}\right)
\left\|u^{h}\right\|_{\dot{B}^{\frac{d}{p}-1}_{p,1}}d\tau\\
&\lesssim& \mathcal{D}^{2}_{p}(t)\int_{0}^{t}\langle t-\tau\rangle^{-\frac{s_{1}+s}{2}} \left(\langle\tau\rangle^{-\frac{s_{1}}{2}-\frac{d}{2p}-\frac{1}{2}}
+\langle\tau\rangle^{-\frac{s_{1}+2-s_{0}}{2}}\right)\langle\tau\rangle^{-\alpha}d\tau.
\end{eqnarray*}
Remembering that $\alpha>1>\frac{d}{p}$ and $\alpha\geq\frac{s_{1}}{2}+\frac{d}{4}+\frac{1}{2}$ for small enough $\varepsilon>0$, it is shown that
\begin{equation*}
\frac{s_{1}+s}{2}\leq\frac{s_{1}}{2}+\frac{d}{4}+\frac{1}{2}<\frac{s_{1}}{2}+\frac{d}{2p}+\frac{1}{2}+\alpha\leq2\alpha, \ \ \ \
\frac{s_{1}+s}{2}\leq\frac{s_{1}+2-s_{0}}{2}+\alpha
\end{equation*}
and
\begin{equation*}
\frac{s_{1}}{2}+\frac{d}{2p}+\frac{1}{2}+\alpha>1-\frac{d}{4}+\frac{d}{2p}+\alpha>1, \ \ \ \ \ \ \
\frac{s_{1}+1-s_{0}}{2}+\alpha>1,
\end{equation*}
since $p\leq \frac{2d}{d-2}$ implies that $\frac{d}{p}\geq\frac{d}{2}-1$. Therefore, we deduce (thanks to \eqref{Eq:3.1}) that
\begin{equation*}
\int_{0}^{t}\langle t-\tau\rangle^{-\frac{s_{1}+s}{2}}\left\|\Lambda^{-1}\mathrm{div}\,(au^{h})\right\|^{\ell}_{\dot{B}^{-s_{1}}_{2,\infty}}d\tau
\lesssim \langle t\rangle^{-\frac{s_{1}+s}{2}}\mathcal{D}^{2}_{p}(t).
\end{equation*}
To bound the term with $k(a)\nabla a^{h}$, we see that, using the composition inequality in Lebesgue spaces and the embeddings $\dot{B}^{\frac{d}{p}}_{2,1}\hookrightarrow L^{p^{*}}$ and $\dot{B}^{s_{0}}_{p,1}\hookrightarrow L^{p^{*}}$ implies
\begin{equation*}
\left\|k(a)\right\|_{L^{p^{*}}}\lesssim\left\|a\right\|_{L^{p^{*}}}\lesssim\left\|a^{\ell}\right\|_{\dot{B}^{\frac{d}{p}}_{2,1}}
+\left\|a^{h}\right\|_{\dot{B}^{s_{0}}_{p,1}}
\lesssim\left\|a^{\ell}\right\|_{\dot{B}^{\frac{d}{p}}_{2,1}}+\left\|a^{h}\right\|_{\dot{B}^{\frac{d}{p}}_{p,1}}.
\end{equation*}
Hence, taking advantage of \eqref{Eq:3.21} and Proposition \ref{Prop2.4},
\begin{equation}\label{Eq:3.24}
\left\|k(a)\nabla a^{h}\right\|^{\ell}_{\dot{B}^{-s_{1}}_{2,\infty}}
\lesssim\left(\left\|a\right\|_{\dot{B}^{1-\frac{d}{p}}_{p,1}}
+\left\|a^{\ell}\right\|_{\dot{B}^{\frac{d}{p}}_{2,1}}+\left\|a^{h}\right\|_{\dot{B}^{\frac{d}{p}}_{p,1}}\right)
\left\|\nabla a^{h}\right\|_{\dot{B}^{\frac{d}{p}-1}_{p,1}}.
\end{equation}
According to \eqref{Eq:3.16} and \eqref{Eq:3.23}, one can bound the term corresponding to $k(a)\nabla a^{h}$ as $\Lambda^{-1}\mathrm{div}\,(au^{h})$.
For the term containing $u\cdot \nabla u^{h}$, we note that, thanks to \eqref{Eq:3.22}, the relation $1-\frac{d}{p}<\frac{d}{p}$ and interpolation,
\begin{equation*}
\left\|u\cdot \nabla u^{h}\right\|^{\ell}_{\dot{B}^{-s_{1}}_{2,\infty}} \lesssim \left(\left\|u^{\ell}\right\|_{\dot{B}^{\frac{d}{p}}_{2,1}}
+\left\|u^{\ell}\right\|_{\dot{B}^{1-\frac{d}{p}}_{p,1}}+\left\|u^{h}\right\|_{\dot{B}^{\frac{d}{p}-1}_{p,1}}\right)
\left\|\nabla u^{h}\right\|_{\dot{B}^{\frac{d}{p}}_{p,1}}.
\end{equation*}
Owing to $\frac{d}{p}-1<1-\frac{d}{p}$ and $\frac{d}{p}\geq\frac{d}{2}-1$, we have by embedding,
\begin{equation}\label{Eq:3.25}
\left\|(a^{\ell},u^{\ell})(\tau)\right\|_{\dot{B}^{\frac{d}{p}}_{2,1}}
\lesssim \left\|(a^{\ell},u^{\ell})(\tau)\right\|_{\dot{B}^{\frac{d}{2}-1}_{2,1}}, \ \
\left\|(a^{\ell},u^{\ell})(\tau)\right\|_{\dot{B}^{1-\frac{d}{p}}_{p,1}}
\lesssim \left\|(a^{\ell},u^{\ell})(\tau)\right\|_{\dot{B}^{\frac{d}{2}-1}_{2,1}}.
\end{equation}
Therefore, we arrive at
\begin{eqnarray*}
&&\int_{0}^{t}\langle t-\tau\rangle^{-\frac{s_{1}+s}{2}}\left\|u\cdot \nabla u^{h}\right\|^{\ell}_{\dot{B}^{-s_{1}}_{2,\infty}}d\tau\\
&\lesssim& \int_{0}^{t}\langle t-\tau\rangle^{-\frac{s_{1}+s}{2}} \left(\left\|u^{\ell}\right\|_{\dot{B}^{\frac{d}{2}-1}_{2,1}}
+\left\|u^{h}\right\|_{\dot{B}^{\frac{d}{p}-1}_{p,1}}\right)
\left\|\nabla u^{h}\right\|_{\dot{B}^{\frac{d}{p}}_{p,1}} d\tau\\
&=& \left(\int_{0}^{1}+ \int_{1}^{t}\right)(\cdots)d\tau\triangleq \tilde{I}_{1}+\tilde{I}_{2}.
\end{eqnarray*}
It just follows from the definition of $\mathcal{E}_{p}(t)$ that
\begin{equation*}
\tilde{I}_{1}\lesssim \langle t\rangle^{-\frac{s_{1}+s}{2}}\mathcal{E}^{2}_{p}(1).
\end{equation*}
Using the fact that, due to $-s_{1}-1<s_{1}<\frac{d}{2}-1<\frac{d}{2}<\frac{d}{2}+1$,
\begin{equation*}
\sup_{\tau \in [0,t]}\langle\tau\rangle^{\frac{s_{1}}{2}+\frac{d}{4}}\left\|a^{\ell} (\tau)\right\|_{\dot{B}^{\frac{d}{2}-1}_{2,1}}
+\sup_{\tau \in [0,t]}\langle\tau\rangle^{\frac{s_{1}}{2}+\frac{d}{4}-\frac{1}{2}}\left\| u^{\ell} (\tau)\right\|_{\dot{B}^{\frac{d}{2}-1}_{2,1}}
\lesssim \mathcal{D}_{p}(t)
\end{equation*}
together with \eqref{Eq:3.16} and \eqref{Eq:3.1}, we thus get, if $t\geq1$,
\begin{equation*}
\tilde{I}_{2}\lesssim \mathcal{D}^{2}_{p}(t)\int_{1}^{t} \langle t-\tau\rangle^{-\frac{s_{1}+s}{2}}
\left(\langle\tau\rangle^{-\frac{s_{1}}{2}-\frac{d}{4}+\frac{1}{2}}+\langle\tau\rangle^{-\alpha}\right)\langle\tau\rangle^{-\alpha}d\tau
\lesssim\langle t\rangle^{-\frac{s_{1}+s}{2}}\mathcal{D}^{2}_{p}(t),
\end{equation*}
since the relations $s_{1}>1-\frac{d}{2}$ and $\alpha>1>\frac{s_{1}}{2}+\frac{d}{4}$ lead to $2\alpha>\frac{s_{1}}{2}+\frac{d}{4}-\frac{1}{2}+\alpha>1$
and $\frac{s_{1}}{2}+\frac{d}{4}-\frac{1}{2}+\alpha>\frac{s_{1}}{2}+\frac{d}{4}+\frac{1}{2}\geq\frac{s_{1}+s}{2}$ for all $s\leq\frac{d}{2}+1$.

In order to bound the terms with $g_{3}(a,u^{h})$, we mimic the procedure leading to \eqref{Eq:3.24} and get
\begin{equation*}
\left\|g_{3}(a,u^{h})\right\|^{\ell}_{\dot{B}^{-s_{1}}_{2,\infty}}
\lesssim\left(\left\|a\right\|_{\dot{B}^{1-\frac{d}{p}}_{p,1}}
+\left\|a^{\ell}\right\|_{\dot{B}^{\frac{d}{p}}_{2,1}}+\left\|a^{h}\right\|_{\dot{B}^{\frac{d}{p}}_{p,1}}\right)\left\| \nabla u^{h}\right\|_{\dot{B}^{\frac{d}{p}}_{p,1}}.
\end{equation*}
Hence, using the fact that $1-\frac{d}{p}<\frac{d}{p}$ and \eqref{Eq:3.25},
\begin{equation*}
\left\|g_{3}(a,u^{h})\right\|^{\ell}_{\dot{B}^{-s_{1}}_{2,\infty}}
\lesssim\left(\left\|a^{\ell}\right\|_{\dot{B}^{\frac{d}{2}-1}_{2,1}}+\left\|a^{h}\right\|_{\dot{B}^{\frac{d}{p}}_{p,1}}\right)\left\| \nabla u^{h}\right\|_{\dot{B}^{\frac{d}{p}}_{p,1}},
\end{equation*}
and one can conclude that exactly as the previous term $u\cdot\nabla u^{h}$ that
\begin{equation*}
\int_{0}^{t}\langle t-\tau\rangle^{-\frac{s_{1}+s}{2}}\left\|g_{3}(a,u^{h})\right\|^{\ell}_{\dot{B}^{-s_{1}}_{2,\infty}}d\tau
\lesssim\langle t\rangle^{-\frac{s_{1}+s}{2}}\left(\mathcal{D}^{2}_{p}(t)+\mathcal{E}^{2}_{p}(t)\right).
\end{equation*}

Lastly, let us consider the term with $g_{4}(a,u^{h})$. Applying \eqref{Eq:2.5} with $\sigma=1-\frac{d}{p}$ yields for any smooth function $K$ vanishing at $0$,
\begin{equation*}
\left\|\nabla K(a)\otimes \nabla u^{h}\right\|^{\ell}_{\dot{B}^{-s_{0}}_{2,\infty}}
\lesssim \left(\left\|\nabla u^{h}\right\|_{\dot{B}^{1-\frac{d}{p}}_{p,1}}
+\sum_{j=j_{0}}^{j_{0}+N_{0}-1}\left\|\dot{ \Delta}_{j}\nabla u^{h}\right\|_{L^{p^{*}}}\right)\left\|\nabla K(a)\right\|_{\dot{B}^{\frac{d}{p}-1}_{p,1}}.
\end{equation*}
As $p^{*}\geq p$, we get from Bernstein inequality that for $j_{0}\leq j<j_{0}+N_{0}$,
\begin{equation*}
\left\|\dot{ \Delta}_{j}\nabla u^{h}\right\|_{L^{p^{*}}}\lesssim \left\|\dot{ \Delta}_{j}\nabla u^{h}\right\|_{L^{p}}.
\end{equation*}
Hence, keeping in mind $s_{1}\leq s_{0}$ and using Proposition \ref{Prop2.4} and the fact that $1-\frac{d}{p}<\frac{d}{p}$, we have
\begin{equation*}
\left\|g_{4}(a,u^{h})\right\|^{\ell}_{\dot{B}^{-s_{1}}_{2,\infty}}
\lesssim\left\|g_{4}(a,u^{h})\right\|^{\ell}_{\dot{B}^{-s_{0}}_{2,\infty}}
\lesssim \left\|a\right\|_{\dot{B}^{\frac{d}{p}}_{p,1}}\left\|\nabla u^{h}\right\|_{\dot{B}^{1-\frac{d}{p}}_{p,1}}
\lesssim \left\|a\right\|_{\dot{B}^{\frac{d}{p}}_{p,1}}\left\|\nabla u^{h}\right\|_{\dot{B}^{\frac{d}{p}}_{p,1}}.
\end{equation*}
Therefore,
\begin{eqnarray*}
\int_{0}^{t}\langle t-\tau\rangle^{-\frac{s_{1}+s}{2}}\left\|g_{4}(a,u^{h})\right\|^{\ell}_{\dot{B}^{-s_{1}}_{2,\infty}}d\tau
&\lesssim&\int_{0}^{t}\langle t-\tau\rangle^{-\frac{s_{1}+s}{2}}\left\|a\right\|_{\dot{B}^{\frac{d}{p}}_{p,1}}\left\|\nabla u^{h}\right\|_{\dot{B}^{\frac{d}{p}}_{p,1}}d\tau \nonumber\\
&=&\left(\int_{0}^{1}+\int_{1}^{t}\right)(\cdots)d\tau
\triangleq  \tilde{J}_{1}+\tilde{J}_{2}.
\end{eqnarray*}
For $\tilde{J}_{1}$, it is clear that
\begin{equation*}
\tilde{J}_{1}\lesssim \langle t\rangle^{-\frac{s_{1}+s}{2}}\mathcal{E}^{2}_{p}(1)
\end{equation*}
 and that, thanks to \eqref{Eq:3.11}, \eqref{Eq:3.16} and \eqref{Eq:3.1}, if $t\geq1$,
\begin{eqnarray*}
\tilde{J}_{2}&\lesssim& \left(\sup_{\tau \in [1,t]}\langle\tau\rangle^{\frac{s_{1}}{2}+\frac{d}{4}+\frac{1}{2}}\left\|a\right\|_{\dot{B}^{\frac{d}{p}}_{p,1}}\right)
\left(\sup_{\tau \in [1,t]}\tau^{\alpha}\left\|\nabla u^{h}\right\|_{\dot{B}^{\frac{d}{p}}_{p,1}}\right)\\
&&\times\int_{1}^{t}\langle t-\tau\rangle^{-\frac{s_{1}+s}{2}}\langle\tau\rangle^{-\frac{s_{1}}{2}-\frac{d}{4}-\frac{1}{2}-\alpha}d\tau
\lesssim \langle t\rangle^{-\frac{s_{1}+s}{2}}\mathcal{D}^{2}_{p}(t).
\end{eqnarray*}
In summary, \eqref{Eq:3.20} can be followed.
\end{proof}

Combining those estimates in Lemmas of \ref{Lem3.1}-\ref{Lem3.3}, we get \eqref{Eq:3.7} eventually . Hence, the proof of Proposition
\ref{Prop3.1} is finished. Furthermore, together with \eqref{Eq:3.6}, we conclude that
\begin{equation} \label{Eq:3.26}
\langle t\rangle^{\frac{s_{1}+s}{2}}\left\|\left(\tilde{a},u\right)(t)\right\|^{\ell}_{\dot{B}^{s}_{2,1}}
\lesssim \mathcal{D}_{p,0}+\mathcal{D}^{2}_{p}(t)+\mathcal{E}^{2}_{p}(t) \ \ \hbox{for all} \ \ t\geq0,
\end{equation}
provided that $-s_{1}<s\leq\frac{d}{2}+1$.

\subsection{Second Step: Decay  Estimates  for  the  High  Frequencies  of  $\left(\nabla a,u\right)$}
In this section, we shall perform the energy approach of $L^{p}$ type in terms of the effective velocity, and bound the second term of $\mathcal{D}_{p}(t)$.
We highlight the new observation and track the optimal exponent for high frequencies of $(a,u)$ with the assumption \eqref{Eq:1.8}-\eqref{Eq:1.9}.

Let $\mathcal{P}\triangleq \mathrm{Id}+\nabla\left(-\Delta\right)^{-1}\mathrm{div}$
be the Leray projector onto divergence-free vector-fields. It follows from \eqref{Eq:1.7}
that $\mathcal{P}u$ fulfills
\begin{equation*}
\partial _{t}\mathcal{P}u-\mu_{\infty} \Delta \mathcal{P}u=\mathcal{P}g.
\end{equation*}
Following from Haspot's approach in \cite{HB} which is originated from Hoff's viscous effective flux \cite{HD}, let us introduce the effective velocity $w$ defined in \eqref{Eq:1.12}.
We observe that $(a,w)$ satisfies
\begin{equation*}
\left\{
\begin{array}{l}
\partial _{t}w-\Delta w=\nabla\left(-\Delta \right)^{-1}\left(f-\mathrm{div}\,g\right) +w-2\left(-\Delta \right)^{-1}\nabla a, \\
\partial _{t}a+a=f-\mathrm{div}\,w.
\end{array}
\right.
\end{equation*}
\begin{prop}\label{Prop3.2}
If $p$ fulfills \eqref{Eq:1.5}, then it holds that for all $T\geq0$,
\begin{equation}\label{Eq:3.27}
\left\|\langle t\rangle^{\alpha}\left(\nabla a,u\right)\right\|^{h}_{\tilde{L}^{\infty}_{T}(\dot{B}^{\frac{d}{p}-1}_{p,1})}
\lesssim \left\|\left(\nabla a_{0},u_{0}\right)\right\|^{h}_{\dot{B}^{\frac{d}{p}-1}_{p,1}}+\mathcal{E}^{2}_{p}(T)+\mathcal{D}^{2}_{p}(T)
\end{equation}
with $\alpha=s_{1}+\frac d2+\frac 12-\varepsilon$ for sufficiently $\varepsilon>0$, where $\mathcal{E}_{p}(T)$ and $\mathcal{D}_{p}(T)$ are defined in \eqref{Eq:1.6} and \eqref{Eq:1.11}, respectively.
\end{prop}
\begin{proof}
By applying the $L^{p}$ energy approach, we can conclude that (see \cite{DX} for details)
\begin{eqnarray}
\left\|\langle t\rangle^{\alpha}\left(\nabla a,u\right)\right\|^{h}_{\tilde{L}^{\infty}_{T}(\dot{B}^{\frac{d}{p}-1}_{p,1})}
&\lesssim& \|(\nabla a_{0},u_{0}\|^{h}_{\dot{B}^{\frac{d}{p}-1}_{p,1}} \nonumber\\
\label{Eq:3.28}
&&+\sum_{j\geq j_{0}-1}\sup_{t\in[0,T]}\left(\langle t\rangle^{\alpha}\int_{0}^{t}e^{-c_{0}(t-\tau)}2^{j(\frac{d}{p}-1)}Z_{j}(\tau)d\tau\right)
\end{eqnarray}
with $Z_{j}\triangleq Z^{1}_{j}+\cdots + Z^{5}_{j}$ and
\begin{eqnarray*}
&&Z^{1}_{j}\triangleq\left\|\dot{\Delta}_{j}\left(au\right)\right\|_{L^{p}},\ \ \ \ \ \ \ \ \ Z^{2}_{j}\triangleq\left\|g_{j}\right\|_{L^{p}},\ \ \ \ \ \ Z^{3}_{j}\triangleq\left\|\nabla \dot{\Delta}_{j}\left(a\,\mathrm{div}\,u\right)\right\|_{L^{p}},\ \ \ \ \ \ \\
&&Z^{4}_{j}\triangleq\left\|R_{j}\right\|_{L^{p}},\ \ \ \ \ \ \  \ \ \ \ \ \ \ \ \ Z^{5}_{j}\triangleq\left\|\mathrm{div}\,u\right\|_{L^{\infty}}\left\|\nabla a_{j}\right\|_{L^{p}},
\end{eqnarray*}
where the notations $a_{j}\triangleq \dot{\Delta}_{j}a$, $g_{j}\triangleq \dot{\Delta}_{j}g$ and $R_{j}\triangleq [u\cdot \nabla,\nabla\dot{\Delta}_{j}]a$.

Firstly, we see that
\begin{equation*}
\sum_{j\geq j_{0}-1}\sup_{t\in[0,2]}\left(\langle t\rangle^{\alpha}\int_{0}^{t}e^{-c_{0}(t-\tau)}2^{j(\frac{d}{p}-1)}Z_{j}(\tau)d\tau\right)
\lesssim \int_{0}^{2}\sum_{j\geq j_{0}-1}2^{j(\frac{d}{p}-1)}Z_{j}(\tau)d\tau.
\end{equation*}
Furthermore, with the aid of Propositions \ref{Prop2.2} and \ref{Prop2.5}, we can obtain
\begin{equation} \label{Eq:3.29}
\int_{0}^{2}\sum_{j\geq j_{0}-1}2^{j(\frac{d}{p}-1)}Z_{j}(\tau)d\tau
\lesssim \int_{0}^{2}\left(\left\|\left(au,g\right)\right\|^{h}_{\dot{B}^{\frac{d}{p}-1}_{p,1}}
+\left\|\nabla u\right\|_{\dot{B}^{\frac{d}{p}}_{p,1}}\left\|a\right\|_{\dot{B}^{\frac{d}{p}}_{p,1}}\right)d\tau.
\end{equation}
It is obvious that the last term of the r.h.s. of \eqref{Eq:3.29} may be bounded by $C\mathcal{E}^{2}_{p}(2)$ and that, due to Proposition \ref{Prop2.2}, we easily get
\begin{equation*}
\left\|au\right\|_{L_{t}^{1}(\dot{B}_{p,1}^{\frac {d}{p}-1})}^{h}
\lesssim \left\|a\right\|_{L_{t}^{1} (\dot{B}_{p,1}^{\frac {d}{p}})}\left\|u\right\|_{L_{t}^{\infty} (\dot{B}_{p,1}^{\frac {d}{p}-1})}.
\end{equation*}
Additionally, applying Propositions  \ref{Prop2.2} and \ref{Prop2.4} yields
\begin{eqnarray*}
\left\|g\right\|_{L_{t}^{1}(\dot{B}_{p,1}^{\frac {d}{p}-1})}^{h}
&\lesssim& \left(\left\|u\right\|_{L_{t}^{\infty} (\dot{B}_{p,1}^{\frac {d}{p}-1})}\left\|\nabla u\right\|_{L_{t}^{1} (\dot{B}_{p,1}^{\frac {d}{p}})}+\left\|a\right\|_{L_{t}^{\infty} (\dot{B}_{p,1}^{\frac {d}{p}})}\left\|\nabla u\right\|_{L_{t}^{1} (\dot{B}_{p,1}^{\frac {d}{p}})}\right.\\
&&\left.+\left\|a\right\|_{L_{t}^{\infty} (\dot{B}_{p,1}^{\frac {d}{p}})}\left\|\nabla a\right\|_{L_{t}^{1} (\dot{B}_{p,1}^{\frac {d}{p}-1})}\right),
\end{eqnarray*}
Hence, using the definition of $\mathcal{E}_{p}(t)$,
we can conclude that the first term of the r.h.s. of \eqref{Eq:3.29} may be bounded by C$\mathcal{E}^{2}_{p}(2)$.
So we finally conclude that
\begin{equation}\label{Eq:3.30}
\sum_{j\geq j_{0}-1}\sup_{t\in[0,2]}\left(\langle t\rangle^{\alpha}\int_{0}^{t}e^{-c_{0}(t-\tau)}2^{j(\frac{d}{p}-1)}Z_{j}(\tau)d\tau\right)
\lesssim \mathcal{E}^{2}_{p}(2).
\end{equation}
Secondly, let us bound the supremum for $2\leq t\leq T$ in the last term of \eqref{Eq:3.28}. To do this, one can split the integral on $[0,t]$ into integrals on $[0,1]$ and $[1,t]$. The $[0,1]$ part of the integral
can be handled by C$\mathcal{E}^{2}_{p}(1)$ exactly as the supremum on $[0,2]$ treated before, see \cite{DX} for more details.
In order to bound the integral on $[1,t]$ for $2\leq t\leq T$, we start from
\begin{equation}\label{Eq:3.31}
\sum_{j\geq j_{0}-1}\sup_{t\in[2,T]}\left(\langle t\rangle^{\alpha}\int_{1}^{t}e^{-c_{0}(t-\tau)}2^{j(\frac{d}{p}-1)}Z_{j}(\tau)d\tau\right)
\lesssim \sum_{j\geq j_{0}-1} 2^{j(\frac{d}{p}-1)}\sup_{t\in[1,T]}t^{\alpha}Z_{j}(t).
\end{equation}
To handle the contribution of $Z^{1}_{j}$ and $Z^{2}_{j}$ in \eqref{Eq:3.28}, we notice that
\begin{equation}\label{Eq:3.32}
\sum_{j\geq j_{0}-1}2^{j(\frac{d}{p}-1)}\sup_{t \in[1,T]} t^{\alpha}\left(Z^{1}_{j}(t)+Z^{2}_{j}(t)\right)
\lesssim \left\|t^{\alpha}(au,g)\right\|^{h}_{\tilde{L}^{\infty}_{T}(\dot{B}^{\frac{d}{p}-1}_{p,1})}.
\end{equation}
For the term with $au$, we use the decomposition
\begin{equation*}
au=au^{h}+a^{h}u^{\ell}+a^{\ell}u^{\ell}.
\end{equation*}
Product laws in Proposition \ref{Prop2.2} adapted to tilde spaces ensure that
\begin{eqnarray*}
&&\left\|t^{\alpha}au^{h}\right\|^{h}_{\tilde{L}^{\infty}_{T}(\dot{B}^{\frac{d}{p}-1}_{p,1})}
\lesssim\left\|a\right\|_{\tilde{L}^{\infty}_{T}(\dot{B}^{\frac{d}{p}}_{p,1})}
\left\|t^{\alpha}u^{h}\right\|_{\tilde{L}^{\infty}_{T}(\dot{B}^{\frac{d}{p}-1}_{p,1})}
\lesssim\mathcal{E}_{p}(T)\mathcal{D}_{p}(T),\\
&&\left\|t^{\alpha}a^{h}u^{\ell}\right\|^{h}_{\tilde{L}^{\infty}_{T}(\dot{B}^{\frac{d}{p}-1}_{p,1})}
\lesssim\left\|t^{\alpha}a^{h}\right\|_{\tilde{L}^{\infty}_{T}(\dot{B}^{\frac{d}{p}}_{p,1})}
\left\|u^{\ell}\right\|_{\tilde{L}^{\infty}_{T}(\dot{B}^{\frac{d}{p}-1}_{p,1})}
\lesssim\mathcal{D}_{p}(T)\mathcal{E}_{p}(T).
\end{eqnarray*}
In what follows, we claim the following two inequalities
\begin{equation}\label{Eq:3.33}
\left\|t^{\frac{s_{1}}{2}+\frac{d}{4}+\frac{1}{2}-\frac{\varepsilon}{2}} (a^{\ell},\nabla u^{\ell})(\tau)\right\|_{\tilde{L}^{\infty}_{T}(\dot{B}^{\frac{d}{p}}_{p,1})}
\lesssim\mathcal{D}_{p}(T), \ \left\|t^{\frac{s_{1}}{2}+\frac{d}{4}-\frac{\varepsilon}{2}}u^{\ell}(\tau)\right\|_{\tilde{L}^{\infty}_{T}(\dot{B}^{\frac{d}{p}}_{p,1})}
\lesssim\mathcal{D}_{p}(T).
\end{equation}
Indeed, it follows from the embedding, the definition of $\mathcal{D}_{p}(t)$ and tilde norms that
\begin{eqnarray*}
 \left\|t^{\frac{s_{1}}{2}+\frac{d}{4}+\frac{1}{2}-\frac{\varepsilon}{2}} (a^{\ell},\nabla u^{\ell})(\tau)\right\|_{\tilde{L}^{\infty}_{T}(\dot{B}^{\frac{d}{p}}_{p,1})}
&\lesssim&
 \left\|t^{\frac{s_{1}}{2}+\frac{d}{4}+\frac{1}{2}-\frac{\varepsilon}{2}} (a^{\ell},\nabla u^{\ell})(\tau)\right\|_{\tilde{L}^{\infty}_{T}(\dot{B}^{\frac{d}{2}}_{2,1})}\\
&\lesssim&  \left\|\langle t\rangle^{\frac{s_{1}}{2}+\frac{d}{4}+\frac{1}{2}-\frac{\varepsilon}{2}} (a^{\ell},\nabla u^{\ell})(\tau)\right\|_{L^{\infty}_{T}(\dot{B}^{\frac{d}{2}-\varepsilon}_{2,1})}\\
&\lesssim&\mathcal{D}_{p}(T),\\
\left\|t^{\frac{s_{1}}{2}+\frac{d}{4}-\frac{\varepsilon}{2}}u^{\ell}(\tau)\right\|_{\tilde{L}^{\infty}_{T}(\dot{B}^{\frac{d}{p}}_{p,1})}
&\lesssim &\left\|t^{\frac{s_{1}}{2}+\frac{d}{4}-\frac{\varepsilon}{2}}u^{\ell}(\tau)\right\|_{\tilde{L}^{\infty}_{T}(\dot{B}^{\frac{d}{2}}_{2,1})}\\
&\lesssim & \left\|\langle t\rangle^{\frac{s_{1}}{2}+\frac{d}{4}-\frac{\varepsilon}{2}} u^{\ell}(\tau)\right\|_{L^{\infty}_{T}(\dot{B}^{\frac{d}{2}-\varepsilon}_{2,1})}\lesssim\mathcal{D}_{p}(T).
\end{eqnarray*}
By using \eqref{Eq:3.33} and Proposition \ref{Prop2.2}, we obtain
\begin{eqnarray*}
\left\|t^{\alpha}a^{\ell}u^{\ell}\right\|^{h}_{\tilde{L}^{\infty}_{T}(\dot{B}^{\frac{d}{p}-1}_{p,1})}
&\lesssim& \left\|t^{\alpha}a^{\ell}u^{\ell}\right\|^{h}_{\tilde{L}^{\infty}_{T}(\dot{B}^{\frac{d}{p}}_{p,1})}\\
&\lesssim& \left\|t^{\frac{s_{1}}{2}+\frac{d}{4}+\frac{1}{2}-\frac{\varepsilon}{2}} a^{\ell}\right\|_{\tilde{L}^{\infty}_{T}(\dot{B}^{\frac{d}{p}}_{p,1})}
\left\|t^{\frac{s_{1}}{2}+\frac{d}{4}-\frac{\varepsilon}{2}}u^{\ell}\right\|_{\tilde{L}^{\infty}_{T}(\dot{B}^{\frac{d}{p}}_{p,1})}
 \lesssim  \mathcal{D}^{2}_{p}(T).
\end{eqnarray*}
With the aid of Proposition \ref{Prop2.2}, \eqref{Eq:3.33}, the definition of $\mathcal{D}_{p}(t)$ and tilde norms, we arrive at
\begin{eqnarray*}
\left\|t^{\alpha}(u\cdot \nabla u^{h})\right\|^{h}_{\tilde{L}^{\infty}_{T}(\dot{B}^{\frac{d}{p}-1}_{p,1})}
&\lesssim& \left\|u\right\|_{\tilde{L}^{\infty}_{T}(\dot{B}^{\frac{d}{p}-1}_{p,1})}
\left\|t^{\alpha} \nabla u^{h}\right\|_{\tilde{L}^{\infty}_{T}(\dot{B}^{\frac{d}{p}}_{p,1})}
\lesssim \mathcal{E}_{p}(T)\mathcal{D}_{p}(T),\\
\left\|t^{\alpha}(u^{h}\cdot \nabla u^{\ell})\right\|^{h}_{\tilde{L}^{\infty}_{T}(\dot{B}^{\frac{d}{p}-1}_{p,1})}
&\lesssim& \left\|t^{\frac{s_{1}}{2}+\frac{d}{4}-\frac{\varepsilon}{2}} u^{h}\right\|_{\tilde{L}^{\infty}_{T}(\dot{B}^{\frac{d}{p}-1}_{p,1})}
\left\|t^{\frac{s_{1}}{2}+\frac{d}{4}+\frac{1}{2}-\frac{\varepsilon}{2}} \nabla u^{\ell}\right\|_{\tilde{L}^{\infty}_{T}(\dot{B}^{\frac{d}{p}}_{p,1})}\\
&\lesssim& \mathcal{D}^{2}_{p}(T)
\end{eqnarray*}
and
\begin{eqnarray*}
\left\|t^{\alpha}(u^{\ell}\cdot \nabla u^{\ell})\right\|^{h}_{\tilde{L}^{\infty}_{T}(\dot{B}^{\frac{d}{p}-1}_{p,1})}
&\lesssim&\left\|t^{\alpha}(u^{\ell}\cdot \nabla u^{\ell})\right\|^{h}_{\tilde{L}^{\infty}_{T}(\dot{B}^{\frac{d}{p}}_{p,1})}\\
&\lesssim& \left\|t^{\frac{s_{1}}{2}+\frac{d}{4}-\frac{\varepsilon}{2}}u^{\ell}\right\|_{\tilde{L}^{\infty}_{T}(\dot{B}^{\frac{d}{p}}_{p,1})}
\left\|t^{\frac{s_{1}}{2}+\frac{d}{4}+\frac{1}{2}-\frac{\varepsilon}{2}} \nabla u^{\ell}\right\|_{\tilde{L}^{\infty}_{T}(\dot{B}^{\frac{d}{p}}_{p,1})}\\
&\lesssim &\mathcal{D}^{2}_{p}(T).
\end{eqnarray*}
In addition, we shall use the fact that, owing to the definition of $\mathcal{D}_{p}(t)$ and $\alpha>\frac{s_{1}}{2}+\frac{d}{4}-\frac{\varepsilon}{2}$ for small enough $\varepsilon>0$,
\begin{eqnarray}\label{Eq:3.34}
\left\|t^{\frac{s_{1}}{2}+\frac{d}{4}-\frac{\varepsilon}{2}}a\right\|_{\tilde{L}^{\infty}_{T}(\dot{B}^{\frac{d}{p}}_{p,1})}
&\lesssim &\left\| t^{\frac{s_{1}}{2}+\frac{d}{4}-\frac{\varepsilon}{2}}a^{\ell}\right\|_{\tilde{L}^{\infty}_{T}(\dot{B}^{\frac{d}{p}}_{p,1})}
+\left\| t^{\frac{s_{1}}{2}+\frac{d}{4}-\frac{\varepsilon}{2}}a^{h}\right\|_{\tilde{L}^{\infty}_{T}(\dot{B}^{\frac{d}{p}}_{p,1})}\nonumber\\
&\lesssim &\left\|\langle t\rangle^{\frac{s_{1}}{2}+\frac{d}{4}+\frac{1}{2}-\frac{\varepsilon}{2}}a^{\ell}\right\|_{L^{\infty}_{T}(\dot{B}^{\frac{d}{2}-\varepsilon}_{2,1})}
+\left\|\langle t\rangle ^{\alpha}a^{h}\right\|_{\tilde{L}^{\infty}_{T}(\dot{B}^{\frac{d}{p}}_{p,1})}\nonumber\\
&\lesssim&  \mathcal{D}_{p}(T).
\end{eqnarray}
To bound the term containing $k(a)\nabla a$, we take advantage of Propositions \ref{Prop2.2}, \ref{Prop2.4} and \eqref{Eq:3.33}, \eqref{Eq:3.34}, and get
\begin{eqnarray*}
\left\|t^{\alpha}k( a) \nabla a^{\ell}\right\|^{h}_{\tilde{L}^{\infty}_{T}(\dot{B}^{\frac{d}{p}-1}_{p,1})}
&\lesssim &\left\|t^{\frac{s_{1}}{2}+\frac{d}{4}-\frac{\varepsilon}{2}}a\right\|_{\tilde{L}^{\infty}_{T}(\dot{B}^{\frac{d}{p}}_{p,1})}
\left\|t^{\frac{s_{1}}{2}+\frac{d}{4}+\frac{1}{2}-\frac{\varepsilon}{2}} \nabla a^{\ell}\right\|_{\tilde{L}^{\infty}_{T}(\dot{B}^{\frac{d}{p}-1}_{p,1})}\\
&\lesssim &\mathcal{D}^{2}_{p}(T),\\
\left\|t^{\alpha}k( a) \nabla a^{h}\right\|^{h}_{\tilde{L}^{\infty}_{T}(\dot{B}^{\frac{d}{p}-1}_{p,1})}
&\lesssim &\left\|a\right\|_{\tilde{L}^{\infty}_{T}(\dot{B}^{\frac{d}{p}}_{p,1})}
\left\|t^{\alpha} \nabla a^{h}\right\|_{\tilde{L}^{\infty}_{T}(\dot{B}^{\frac{d}{p}-1}_{p,1})}
\lesssim \mathcal{E}_{p}(T)\mathcal{D}_{p}(T).
\end{eqnarray*}
The term $g_{3}(a,u)$ of $g$ is of the type $H(a)\nabla^{2} a$ with $H(a)=0$, and the term $g_{4}(a,u)$ of $g$ is of the type $\nabla K(a)\otimes \nabla u$ with $K(0)=0$. Consequently, we write that thanks to Propositions \ref{Prop2.2} and \ref{Prop2.4}, and \eqref{Eq:3.33}, \eqref{Eq:3.34},
\begin{eqnarray*}
&&\left\|t^{\alpha}H(a)\nabla^{2}u^{\ell}\right\|^{h}_{\tilde{L}^{\infty}_{T}(\dot{B}^{\frac{d}{p}-1}_{p,1})}
+\left\|t^{\alpha}\nabla K(a)\otimes\nabla u^{\ell}\right\|^{h}_{\tilde{L}^{\infty}_{T}(\dot{B}^{\frac{d}{p}-1}_{p,1})}\\
&\lesssim&\left\|t^{\frac{s_{1}}{2}+\frac{d}{4}-\frac{\varepsilon}{2}}a\right\|_{\tilde{L}^{\infty}_{T}(\dot{B}^{\frac{d}{p}}_{p,1})}
\left\|t^{\frac{s_{1}}{2}+\frac{d}{4}+\frac{1}{2}-\frac{\varepsilon}{2}} \nabla u^{\ell}\right\|_{\tilde{L}^{\infty}_{T}(\dot{B}^{\frac{d}{p}}_{p,1})}
\lesssim \mathcal{D}^{2}_{p}(T),\\
&&\left\|t^{\alpha}H(a)\nabla^{2}u^{h}\right\|^{h}_{\tilde{L}^{\infty}_{T}(\dot{B}^{\frac{d}{p}-1}_{p,1})}
+\left\|t^{\alpha}\nabla K(a)\otimes\nabla u^{h}\right\|^{h}_{\tilde{L}^{\infty}_{T}(\dot{B}^{\frac{d}{p}-1}_{p,1})}\\
&\lesssim&\left\|a\right\|_{\tilde{L}^{\infty}_{T}(\dot{B}^{\frac{d}{p}}_{p,1})}
\left\|t^{\alpha} \nabla u^{h}\right\|_{\tilde{L}^{\infty}_{T}(\dot{B}^{\frac{d}{p}}_{p,1})}
\lesssim \mathcal{E}_{p}(T)\mathcal{D}_{p}(T).
\end{eqnarray*}
Reverting to \eqref{Eq:3.32}, we end up with
\begin{equation*}
\sum_{j\geq j_{0}-1}2^{j(\frac{d}{p}-1)}\sup_{t \in[1,T]} t^{\alpha}\left(Z^{1}_{j}(t)+Z^{2}_{j}(t)\right)
\lesssim \mathcal{D}_{p}(T)\mathcal{E}_{p}(T)+\mathcal{D}^{2}_{p}(T).
\end{equation*}
The term with $Z^{3}_{j}$ is similar to the terms $g_{3}(a,u)$ and $g_{4}(a,u)$ of $g$. In fact, owing to \eqref{Eq:3.33} and \eqref{Eq:3.34}, we have
\begin{eqnarray*}
\sum_{j\geq j_{0}-1}2^{j(\frac{d}{p}-1)}\sup_{t \in[1,T]} t^{\alpha} Z^{3}_{j}(t)
&\lesssim&\left\|t^{\frac{s_{1}}{2}+\frac{d}{4}-\frac{\varepsilon}{2}}a\right\|_{\tilde{L}^{\infty}_{T}(\dot{B}^{\frac{d}{p}}_{p,1})}
\left\|t^{\frac{s_{1}}{2}+\frac{d}{4}+\frac{1}{2}-\frac{\varepsilon}{2}} \mathrm{div}\,u^{\ell}\right\|_{\tilde{L}^{\infty}_{T}(\dot{B}^{\frac{d}{p}}_{p,1})} \\
&&+\left\|a\right\|_{\tilde{L}^{\infty}_{T}(\dot{B}^{\frac{d}{p}}_{p,1})}
\left\|t^{\alpha} \mathrm{div}\,u^{h}\right\|_{\tilde{L}^{\infty}_{T}(\dot{B}^{\frac{d}{p}}_{p,1})} \\
&\lesssim &\mathcal{D}^{2}_{p}(T)+\mathcal{E}_{p}(T)\mathcal{D}_{p}(T).
\end{eqnarray*}
For the term $Z^{4}_{j}$, we use a small modification
of the commutator estimate as in Proposition \ref{Prop2.5} (just include $t^{\alpha}$ in the definition of the
commutator, follow the proof treating the time variable as a parameter, and take the supremum
on $[0,T]$ at the end) and get
\begin{eqnarray}
\sum_{j\geq j_{0}-1}2^{j(\frac{d}{p}-1)}\sup_{t\in [0,T]} t^{\alpha}\left\|R_{j}(t)\right\|_{L^{p}}
&\lesssim&\left\|t^{\frac{s_{1}}{2}+\frac{d}{4}+\frac{1}{2}-\frac{\varepsilon}{2}} \nabla u^{\ell}\right\|_{\tilde{L}^{\infty}_{T}(\dot{B}^{\frac{d}{p}}_{p,1})}
\left\|t^{\frac{s_{1}}{2}+\frac{d}{4}-\frac{\varepsilon}{2}} a\right\|_{\tilde{L}^{\infty}_{T}(\dot{B}^{\frac{d}{p}}_{p,1})} \nonumber \\
\label{Eq:3.36}
&&+\left\|t^{\alpha} \nabla u^{h}\right\|_{\tilde{L}^{\infty}_{T}(\dot{B}^{\frac{d}{p}}_{p,1})}
\left\|a\right\|_{\tilde{L}^{\infty}_{T}(\dot{B}^{\frac{d}{p}}_{p,1})}.
\end{eqnarray}
Hence using \eqref{Eq:3.33}, \eqref{Eq:3.34} and the definitions of $\mathcal{E}(t)$ and $\mathcal{D}(t)$ give
\begin{equation*}
\sum_{j\geq j_{0}-1}2^{j(\frac{d}{p}-1)}\sup_{t\in [0,T]} t^{\alpha}\left\|R_{j}(t)\right\|_{L^{p}}
\lesssim \mathcal{D}^{2}_{p}(T)+\mathcal{D}_{p}(T)\mathcal{E}_{p}(T).
\end{equation*}
The term with $Z^{5}_{j}$ is clearly handled by the r.h.s. of \eqref{Eq:3.36}.
Putting all the above estimates together, we discover that
\begin{equation}\label{Eq:3.37}
\sum_{j\geq j_{0}-1}2^{j(\frac{d}{p}-1)}\sup_{t \in[1,T]} t^{\alpha}Z_{j}(t)
\lesssim \mathcal{E}_{p}(T)\mathcal{D}_{p}(T)+\mathcal{D}^{2}_{p}(T).
\end{equation}
Plugging \eqref{Eq:3.37} in \eqref{Eq:3.31}, and remembering \eqref{Eq:3.30} and \eqref{Eq:3.28}, we end up with \eqref{Eq:3.27}.
 This completes the proof of Proposition \ref{Prop3.2}.
\end{proof}
\subsection{Third Step: Decay  and  Gain  of  Regularity  for  the  High  Frequencies  of  $u$ }
The last step is devoted to handle the third norm in the functional $\mathcal{D}_{p}(t)$ and close the
whole time-weighted inequality. Precisely, one has
\begin{prop}\label{Prop3.3}
If $p$ fulfills \eqref{Eq:1.5}, then it holds that for all $t\geq0$,
\begin{equation}\label{Eq:3.38}
\left\|t^{\alpha}\nabla u\right\|^{h}_{\tilde{L}^{\infty}_{t}(\dot{B}^{\frac{d}{p}}_{p,1})}
\lesssim \left\|\left(\nabla a_{0},u_{0}\right)\right\|^{h}_{\dot{B}^{\frac{d}{p}-1}_{p,1}}
+\mathcal{E}^{2}_{p}(t)+\mathcal{D}^{2}_{p}(t)
\end{equation}
with $\alpha=s_{1}+\frac d2+\frac 12-\varepsilon$ for sufficiently $\varepsilon>0$, where $\mathcal{E}_{p}(T)$ and $\mathcal{D}_{p}(T)$ are defined in \eqref{Eq:1.6} and \eqref{Eq:1.11}, respectively.
\end{prop}
\begin{proof}
In comparison with \cite{DX}, a new observation is involved in the proof, which enables us to obtain more decay in time for the improved regularity of velocity.
Precisely, we see that the velocity $u$ fulfills
\begin{equation}\label{Eq:3.39}
\partial _{t}u-\mathcal{A}u=g-\nabla a-\nabla (-\Delta)^{-1}a.
\end{equation}
To obtain the desired estimates, we reformulate \eqref{Eq:3.39} as follows:
\begin{equation*}
\left\{
\begin{array}{l}
\partial _{t}\left(t^{\alpha}\mathcal{A}u\right)-\mathcal{A}\left(t^{\alpha}\mathcal{A}u\right)=\alpha t^{\alpha-1}\mathcal{A}u+t^{\alpha}\mathcal{A}g\\
\hspace{2cm}-t^{\alpha}\mathcal{A}\nabla a-t^{\alpha}\mathcal{A} \nabla (-\Delta)^{-1}a, \\
t^{\alpha}\mathcal{A}u|_{t=0}=0.
\end{array}
\right.
\end{equation*}
According to the smoothing property of Lam\'{e} semi-group in Remark \ref{Rem2.2}, we have for $j\geq j_{0}-1$,
\begin{eqnarray*}
\left\|\tau^{\alpha} \nabla^{2}u\right\|^{h}_{\tilde{L}^{\infty}_{t}(\dot{B}^{\frac{d}{p}-1}_{p,1})}
&\lesssim&\left\|\tau^{\alpha}g\right\|^{h}_{\tilde{L}^{\infty}_{t}(\dot{B}^{\frac{d}{p}-1}_{p,1})}
+\left\|\tau^{\alpha-1}u\right\|^{h}_{\tilde{L}^{\infty}_{t}(\dot{B}^{\frac{d}{p}-1}_{p,1})}\\
&&+\left\|\tau^{\alpha} \nabla a\right\|^{h}_{\tilde{L}^{\infty}_{t}(\dot{B}^{\frac{d}{p}-1}_{p,1})}
+\left\|\tau^{\alpha} \nabla a\right\|^{h}_{\tilde{L}^{\infty}_{t}(\dot{B}^{\frac{d}{p}-3}_{p,1})}.
\end{eqnarray*}
As $\alpha>1$, we discover that
\begin{eqnarray*}
&&\left\|\tau^{\alpha-1}u\right\|^{h}_{\tilde{L}^{\infty}_{t}(\dot{B}^{\frac{d}{p}-1}_{p,1})}
\lesssim \left\|\langle\tau\rangle^{\alpha}u\right\|^{h}_{\tilde{L}^{\infty}_{t}(\dot{B}^{\frac{d}{p}-1}_{p,1})},\\
&&\left\|\tau^{\alpha} \nabla a\right\|^{h}_{\tilde{L}^{\infty}_{t}(\dot{B}^{\frac{d}{p}-1}_{p,1})}
+\left\|\tau^{\alpha} \nabla a\right\|^{h}_{\tilde{L}^{\infty}_{t}(\dot{B}^{\frac{d}{p}-3}_{p,1})}
\lesssim \left\|\langle\tau\rangle^{\alpha}\nabla a\right\|^{h}_{\tilde{L}^{\infty}_{t}(\dot{B}^{\frac{d}{p}-1}_{p,1})}.
\end{eqnarray*}
Furthermore, we arrive at
\begin{eqnarray*}
\left\|\tau^{\alpha} \nabla u\right\|^{h}_{\tilde{L}^{\infty}_{t}(\dot{B}^{\frac{d}{p}}_{p,1})}
\lesssim\left\|\tau^{\alpha}g\right\|^{h}_{\tilde{L}^{\infty}_{t}(\dot{B}^{\frac{d}{p}-1}_{p,1})}
+\left\|\langle\tau\rangle^{\alpha}(\nabla a,u)\right\|^{h}_{\tilde{L}^{\infty}_{t}(\dot{B}^{\frac{d}{p}-1}_{p,1})}.
\end{eqnarray*}
According to \eqref{Eq:3.27}, the last norm can be bounded by
\begin{equation*}
\left\|\left(\nabla a_{0},u_{0}\right)\right\|^{h}_{\dot{B}^{\frac{d}{p}-1}_{p,1}}+\mathcal{E}^{2}_{p}(t)+\mathcal{D}^{2}_{p}(t).
\end{equation*}
Bounding the norm $\left\|\tau^{\alpha}g\right\|^{h}_{\tilde{L}^{\infty}_{t}(\dot{B}^{\frac{d}{p}-1}_{p,1})}$ is exactly same as Step 2, one can deduce that \eqref{Eq:3.38} readily.
\end{proof}
Finally, adding up \eqref{Eq:3.38} to \eqref{Eq:3.26} and \eqref{Eq:3.27} yields for all $T\geq0$,
\begin{equation*}
\mathcal{D}_{p}(T)\lesssim \mathcal{D}_{p,0}
+\left\|\left(\nabla a_{0},u_{0}\right)\right\|^{h}_{\dot{B}^{\frac{d}{p}-1}_{p,1}}
+\mathcal{E}^{2}_{p}(T)+\mathcal{D}^{2}_{p}(T).
\end{equation*}
As Theorem \ref{Thm1.1} implies that $\mathcal{E}_{p}\lesssim\mathcal{E}_{p,0}\ll1$
and as
\begin{equation*}
\left\|\left( a_{0},\nabla u_{0}\right) \right\|_{\dot{B}_{2,1}^{\frac {d}{2}-2}}^{\ell}\lesssim \left\|(\tilde{a}_{0}, u_{0})\right\|^{\ell}_{\dot{B}_{2,\infty}^{-s_{1}}},
\end{equation*}
one can conclude that \eqref{Eq:1.10} is fulfilled for all time if $\mathcal{D}_{p,0}$ and $\left\|\left(\nabla a_{0},u_{0}\right)\right\|^{h}_{\dot{B}^{\frac{d}{p}-1}_{p,1}}$ are small enough. This completes the proof of Theorem \ref{Thm1.2}.

\subsection{The proof of Corollary \ref{Cor1.1}}
\begin{proof}
Recall that for functions with compactly supported Fourier transform, one has the embedding $\dot{B}^{s}_{2,1}\hookrightarrow\dot{B}_{p,1}^{s-d(\frac{1}{2}-\frac{1}{p})}\hookrightarrow\dot{B}^{s}_{p,1}$
for the low frequencies, whenever $p\geq2$. Hence, we may write
\begin{equation*}
\sup_{t\in[0,T]}\langle t\rangle^{\frac{s_{1}+s}{2}}\left\|\Lambda^{s}u\right\|_{\dot{B}^{0}_{p,1}}
\lesssim\left\|\langle t\rangle^{\frac{s_{1}+s}{2}}u\right\|^{\ell}_{L^{\infty}_{T}(\dot{B}^{s}_{2,1})}+\left\|\langle t\rangle^{\frac{s_{1}+s}{2}}u\right\|^{h}_{L^{\infty}_{T}(\dot{B}^{s}_{p,1})}.
\end{equation*}
Taking advantage of \eqref{Eq:1.10} and the definition of functional $\mathcal{D}_{p}(t)$, we see that
\begin{equation*}
\left\|\langle t\rangle^{\frac{s_{1}+s}{2}}u\right\|^{\ell}_{L^{\infty}_{T}(\dot{B}^{s}_{2,1})}
\lesssim\mathcal{D}_{p,0}+\left\|\left(\nabla a_{0},u_{0}\right)\right\|^{h}_{\dot{B}^{\frac{d}{p}-1}_{p,1}} \ \hbox{if} \  -s_{1}<s\leq\frac{d}{2}+1.
\end{equation*}
On the other hand, if $\varepsilon>0$ is sufficiently small, then we have $\frac {s_{1}}{2}+\frac d2-\varepsilon>\frac 12+\frac d4-\varepsilon\geq \frac{d}{2p}$
for $p\geq2$, which ensures that $\alpha\geq\frac{s_{1}+s}{2}$. Consequently, we deduce that  for all $s\leq\frac{d}{p}+1$
\begin{equation*}
\left\|\langle t\rangle^{\frac{s_{1}+s}{2}}u\right\|^{h}_{L^{\infty}_{T}(\dot{B}^{s}_{p,1})}
\lesssim\left\| t^{\frac{s_{1}+s}{2}}u\right\|^{h}_{\tilde{L}^{\infty}_{T}(\dot{B}^{s}_{p,1})}
\lesssim\mathcal{D}_{p,0}+\left\|\left(\nabla a_{0},u_{0}\right)\right\|^{h}_{\dot{B}^{\frac{d}{p}-1}_{p,1}},
\end{equation*}
since one can assume $t\geq1$ in the large time asymptotics. This yields the desired result for velocity $u$. Bounding $a$ works almost the same, except that the choice of derivative index $s\leq\frac{d}{p}$. Hence, the proof of Corollary \ref{Cor1.1} is complete.
\end{proof}
\subsection{The proof of Corollary \ref{Cor1.2}}
\begin{proof}
Let us first recall the following Gagliardo-Nirenberg type inequality in \cite{SS} (or see \cite{BCD,XK1,XK2}):
$$
\|\Lambda ^{\beta}f\|_{L^{r}}\lesssim \|\Lambda^{\beta_{1}}f\|^{1-\theta}_{L^{q}}\|\Lambda^{\beta_{2}}f\|^{\theta}_{L^{q}},
$$
where $0\leq\theta\leq1$, $1\leq q\leq r\leq\infty$ and
$$
\beta+d\left(\frac{1}{q}-\frac{1}{r}\right)=\beta_{1}(1-\theta)+\beta_{2}\theta.
$$
For brevity, it suffices to establish those time-optimal decay rates of $u$.
It follows from Corollary \ref{Cor1.1} with $p=2$, and the Gagliardo-Nirenberg inequality with $q=2$ and $\beta=k$ that
\begin{eqnarray*}
\left\|\Lambda ^{k}u\right\|_{L^{r}}&\lesssim& \left\|\Lambda ^{\beta_{1}}u\right\|^{1-\theta}_{L^{2}}\left\|\Lambda^{\beta _{2}}u\right\|^{\theta}_{L^{2}}\\
&\lesssim& \left(\mathcal{D}_{2,0}+\left\|(\nabla a_{0},u_{0})\right\|^{h}_{\dot{B}^{\frac{d}{2}-1}_{2,1}}\right)
\left\{\langle t\rangle^{-\frac{s_{1}}{2}-\frac{\beta_{1}}{2}}\right\}^{1-\theta}\left\{\langle t\rangle^{-\frac{s_{1}}{2}-\frac{\beta_{2}}{2}}\right\}^{\theta}\\
&=&\left(\mathcal{D}_{2,0}+\left\|(\nabla a_{0},u_{0})\right\|^{h}_{\dot{B}^{\frac{d}{2}-1}_{2,1}}\right)
\langle t\rangle^{-\frac{s_{1}}{2}-\frac{\beta_{1}(1-\theta)+\beta_{2}\theta}{2}}\\
&=&\left(\mathcal{D}_{2,0}+\left\|(\nabla a_{0},u_{0})\right\|^{h}_{\dot{B}^{\frac{d}{2}-1}_{2,1}}\right)
\langle t\rangle^{-\frac{s_{1}}{2}-\frac{d}{2}(\frac{1}{2}-\frac{1}{r})-\frac{k}{2}},
\end{eqnarray*}
where we used that
\begin{equation}\label{Eq:3.40}
k+d\left(\frac{1}{2}-\frac{1}{r}\right)=\beta_{1}(1-\theta)+\beta_{2}\theta
\end{equation}
for $-s_{1}<\beta_{1},\beta_{2}\leq\frac{d}{2}+1$. Furthermore, the derivative indices fulfill
\begin{equation*}
-s_{1}<k+d\left(\frac{1}{2}-\frac{1}{r}\right)\leq\frac{d}{2}+1,
\end{equation*}
so the proof of Corollary \ref{Cor1.2} is finished.
\end{proof}
\section*{Acknowledgments}
The second author (J. Xu) is partially supported by the National
Natural Science Foundation of China (11471158) and the Fundamental Research Funds for the Central
Universities (NE2015005). Last but not least, he is grateful to Professor Hai-Liang Li for the suggestion on this question.


\begin{thebibliography}{99}
\bibitem{BCD}
H.~Bahouri, J.~Y.~Chemin and R.~Danchin, \textit{Fourier analysis and
nonlinear partial differential equations}, Grundlehren der mathematischen
Wissenschaften, Vol. {\bf{343}} (Springer, Berlin 2011).

\bibitem{BWY}
Q.~Y.~Bie, Q.~R.~Wang and Z.~A.~Yao, Optimal decay rate for the
compressible Navier-Stokes-Poisson system in the critical $L^{p}$ framework,
\textit{J. Differential Equations} {\bf{263}} (2017) 8391--8417.

\bibitem{CJY}
J.~Y.~Chemin, Th\'{e}or\`{e}mes d'unicit\'{e} pour le syst\`{e}%
m de Navier-Stokes tridimensionnel, \textit{J. Amal. Math.} {\bf{77}} (1999) 27--50.

\bibitem{CM}
M.~Cannone, A generalization of a theorem by Kato on Navier-Stokes equations, \textit{Rev. Mat.
Iberoamericana} {\bf{13}} (1997) 515--542.

\bibitem{CD1}
F.~Charve and R.~Danchin, A global existence result for the
compressible Navier-Stokes equations in the critical $L^{p}$ framework, \textit{Arch.
Ration. Mech. Anal.} {\bf{198}} (2010) 233--271.

\bibitem{CD2}
N.~Chikami and R.~Danchin, On the global existence and time decay estimates in critical spaces for the Navier-Stokes-Poisson system,
\textit{Math. Nachr.}  {\bf{290}} (2017) 1939-1970.

\bibitem{CL}
J.~Y.~Chemin, N.~Lerner, Flot de champs de vecteurs no lipschitziens
et \'{e}quations de Navier-Stokes, \textit{J. Differential Equations} {\bf{248}} (2010) 2130--2170.

\bibitem{CMZ}
Q.~L.~Chen, C.~X.~Miao and Z.~F.~Zhang, Global well-posedness for
compressible Navier-Stokes equations with highly oscillating initial
velocity, \textit{Commun. Pur. Appl. Math.}  {\bf{63}} (2010) 1173--1224.

\bibitem{DR1}
R.~Danchin, Global existence in critical spaces for
compressible Navier-Stokes equations, \textit{Invent. Math.} {\bf{141}} (2000) 579--614.

\bibitem{DX}
R.~Danchin and J.~Xu, Optimal time-decay estimates for the
compressible Navier-Stokes equations in the critical $L^{p}$ framework, \textit{Arch. Ration. Mech. Anal.} {\bf{224}} (2017) 53--90.

\bibitem{FK}
H.~Fujita and T.~Kato, On the Navier-Stokes initial value problem I, \textit{Arch. Ration. Mech. Anal.}
{\bf{16}} (1964) 269--315.

\bibitem{HB}
B.~Haspot, Existence of global strong solutions in critical
spaces for barotropic viscous fluids, \textit{Arch. Ration. Mech. Anal.} {\bf{202}} (2011)
427--460.

\bibitem{HD}
D.~Hoff, Global solutions of the Navier-Stokes equations for
multidimensional compressible flow with discontinuous initial data, \textit{J. Differential Equations}  {\bf{120}} (1995) 215--254.

\bibitem{HL}
C.~C.~Hao and H.~L.~Li, Global existence for compressible Navier-Stokes-Poisson equations in three and higher dimensions,
\textit{J. Differential Equations} {\bf{246}} (2009) 4791--4812.

\bibitem{KY}
H.~Kozono and M.~Yamazaki, Semilinear heat equations and the Navier-Stokes equations with
distributions in new function spaces as initial data, \textit{Comm. Part. Differ. Equ.} {\bf{19}} (1994) 959--1014.

\bibitem{LMZ}
H.~L.~Li, A.~Matsumura and G.~J.~Zhang, Optimal decay rate of the compressible
Navier-Stokes-Poisson system in $\mathbb{R}^{3}$, \textit{Arch. Ration. Mech. Anal.} {\bf {196}} (2010) 681--713.

\bibitem{LZ}
H.~L.~Li and T.~Zhang, Large time behavior of solutions to 3D compressible Navier-Stokes-Poisson system, \textit{Sci. China
Math.} {\bf{55}} (2012) 159--177.

\bibitem{MRS}
P.~A.~Markowich, C.~A.~Ringhofer and C.~Schmeiser, \textit{Semiconductor equations}
(Springer-Verlag, New York 1990).

\bibitem{SS}
V.~Sohinger and R.~M.~Strain, The Boltzmann equation, Besov spaces, and optimal time decay rates in $\mathbb{R}^{n}_{x}$, \textit{Adv. Math.}
{\bf 261} (2014) 274--332.

\bibitem{WYJ}
Y.~J.~Wang, Decay of the Navier-Stokes-Poisson equations,
\textit{J. Differential Equations} {\bf{253}} (2012) 273--297.

\bibitem{WW1}
W.~K.~Wang and Z.~G.~Wu, Pointwise estimates of solution for the Navier-Stokes-Poisson
equations in multidimensions, \textit{J. Differential Equations} {\bf{248}} (2010) 1617--1636.

\bibitem{WW2}
Y.~Z.~Wang and K.~Y.~Wang, Asymptotic behavior of classical solutions to
the compressible Navier-Stokes-Poisson equations in three and higher dimensions, \textit{J. Differential Equations} {\bf{259}} (2015) 25--47.

\bibitem{X1}
J.~Xu, Relaxation-time limit in the isothermal hydrodynamic model for semiconductors,
\textit{SIAM J. Math. Anal.} {\bf 40} (2009), 1979--1991.

\bibitem{XK1}
J.~Xu and S.~Kawashima, Frequency-localization Duhamel principle and its application to the optimal decay of dissipative systems in low dimensions,
\textit{J. Differential Equations} {\bf 261} (2016) 2670--2701.

\bibitem{XK2}
J.~Xu and S.~Kawashima, The optimal decay estimates on the framework of Besov spaces for generally dissipative systems, \textit{Arch. Ration. Mech. Anal.} {\bf 218} (2015) 275--315.
\end{thebibliography}
\end{document}